\documentclass{article}
\usepackage{geometry}[top=15truemm, bottom=15truemm, right=10truemm, left=10truemm]
\usepackage{amsmath,amssymb,amsthm}
\usepackage{amsmath,amssymb,amsthm}
\usepackage{mathabx}
\usepackage[all]{xy}
\usepackage{graphicx}
\usepackage{mathrsfs}
\usepackage{lineno}

\newcommand{\fsets}{{\bf Sets}_f}

\newcommand{\Cl}{\mathscr{B}}

\newcommand{\D}{\mathscr{D}}
\newcommand{\C}{\mathscr{C}}

\newcommand{\FN}{\mathscr{X}}
\newcommand{\FNO}{\mathscr{O}}
\newcommand{\FNm}{\mathfrak{m}}

\newcommand{\Et}{\mathbf{Et}}

\newcommand{\cl}{\mathscr{B}_f}

\newcommand{\idp}{\mathfrak{p}}
\newcommand{\idq}{\mathfrak{q}}

\newcommand{\idf}{\mathfrak{f}}
\newcommand{\ide}{\mathfrak{e}}
\newcommand{\ida}{\mathfrak{a}}
\newcommand{\idb}{\mathfrak{b}}

\newcommand{\idP}{\mathfrak{P}}

\newcommand{\comp}{\mathbb{C}}
\newcommand{\fld}{\mathbb{F}}

\newcommand{\Ad}{\mathbb{A}}

\newcommand{\F}{\mathrm{F}}

\newcommand{\JJ}{\mathcal{J}}
\newcommand{\RR}{\mathcal{R}}
\newcommand{\LL}{\mathcal{L}}
\newcommand{\HH}{\mathcal{H}}

\newcommand{\V}{\mathcal{V}}

\newcommand{\End}{\mathrm{End}}

\newcommand{\Hom}{\mathrm{Hom}}

\newcommand{\frob}{\mathrm{Frob}}

\newcommand{\nat}{\mathbb{N}}
\newcommand{\idd}{\mathfrak{d}}

\theoremstyle{plain}
\newtheorem{thm}{Theorem}[section]
\newtheorem{lem}[thm]{Lemma}
\newtheorem{prop}[thm]{Proposition}

\theoremstyle{definition}
\newtheorem{defn}[thm]{Definition}

\newtheorem{rem}[thm]{Remark}

\numberwithin{equation}{section}
\title{Deformation of the absolute Galois groups of number fields}
\author{Takeo Uramoto\\ Graduate School of Science and Engineering, Kagoshima University}
\date{}
\begin{document}
\maketitle
\begin{abstract}
\noindent
The technical goal of this paper is to construct and study profinite monoids $\D G_K$ for number fields $K$ such that (1) the unit group $\D G_K^\times$ is isomorphic to the absolute Galois group $G_K$; (2) the maximal abelian quotient of $\D G_K$ is isomorphic to the Deligne-Ribet monoid $DR_K$; (3) the idempotents of $\D G_K$ bijectively correspond to subsets of the set $P_K$ of (finite) primes of $K$; (4) the maximal Galois groups $G_{K, S}$ with restricted ramifications $S$ all appear as maximal closed subgroups of $\D G_K$ at idempotents; (5) all maximal closed subgroups of $\D G_K$ at idempotents are of this form. We also discuss the relationship between $\D G_K$ and the $\idp$-typical Witt vectors of Borger and de Smit \cite{Borger_Smit08, Borger_Smit11}, where we will describe the fundamental monoid of the semi-galois category of $\Lambda_{O_\idp}$-rings in \cite{Borger_Smit11} in terms of fields of norms of the local field $K_\idp$ in particular. 
\end{abstract}
\section{Introduction}
\label{s1}
To begin with this paper, let us mean by a \emph{deformation of a profinite group $G$} in general a profinite monoid $D$ whose unit group $D^\times$ is isomorphic to $G$. Our motivating example of this terminology is the \emph{Deligne-Ribet monoid} $DR_K$ associated to each number field $K$, whose unit group is isomorphic to the maximal abelian Galois group $G_K^{ab}$ (cf.\ \S \ref{s2}). 
Originally, $DR_K$ was introduced by Deligne and Ribet \cite{Deligne_Ribet} in their work on $p$-adic interpolations of special values of $L$-functions, and played a pivotal role in the works of Borger and de Smit \cite{Borger_Smit08, Borger_Smit11}, where they revealed an intrinsic connection between classical class field theory and the theory of generalized Witt vectors \cite{Borger1, Borger2}. In our previous works \cite{Uramoto21, Uramoto23}, we then showed that the Deligne-Ribet monoid $DR_K$ can be utilized to ``deform'' Shimura's reciprocity law for Siegel modular functions in the sense that, while Shimura's reciprocity law only relates the special values of Siegel modular functions living at the same level by Galois conjugates, our \emph{deformed} reciprocity law relates those living at different levels too by some higher congruences (cf.\ \S 1 \cite{Uramoto23}); in fact, we proved a canonical correspondence, as \emph{modularity theorem} \cite{Uramoto21, Uramoto23}, between certain distributions of special values of Siegel modular functions (defined as \emph{modular vectors}) and algebraic Witt vectors, where the analysis of the semigroup structure of $DR_K$ plays a fundamental role (cf.\ \S 4.3 \cite{Uramoto21}, \S 4.2 \cite{Uramoto23}). 

The major concern of this paper is to re-interpret these results from more generic perspective; in particular, we focus on the semigroup structure of the Deligne-Ribet monoid $DR_K$ and highlight its relationship to Hilbert's classical ramification theory, so that it admits non-abelian generalizations in a natural way. 

Indeed, $DR_K$ had a notable feature that its semigroup structure reflects the filtration of higher ramification subgroups of $G_K^{ab}$ in the sense proved in \S \ref{s2}; in particular, the maximal abelian Galois groups $G_{K, S}^{ab}$ with restricted ramifications appear as the maximal closed subgroups at idempotents of $DR_K$\footnote{For a profinite monoid $D$, we mean by the \emph{maximal closed subgroup at an idempotent $e \in D$} the unit group $(e D e)^\times$ of the subsemigroup $e D e \subseteq D$. Unlike profinite groups, a profinite monoid $D$ can have non-trivial idempotents $e$; and around them, (possibly non-isomorphic) profinite groups $(e D e)^\times$ appear as subgroups, or precisely, as sub-\emph{semigroups} in that their identity elements ($= e$) are not generally the same as the identity element $1$ of $D$. (That is, the inclusion $(e D e)^\times \hookrightarrow D$ does not preserve the identity element.)}. In other words, these profinite groups $G_{K, S}^{ab}$ could be placed together as \emph{sub-semigroups} (rather than quotients) within the single ambient profinite monoid $DR_K$, and transformed to each other by non-invertible Frobenius operators in $DR_K$ (\S \ref{s2}); this suggests us a new way to control the variation of the groups $G_{K, S}^{ab}$. Moreover, by involving the groups $G_{K, S}^{ab}$ as subsemigroups in $DR_K$, the ambient monoid $DR_K$ gets dual to algebraic Witt vectors in the sense proved in \S 3.1 \cite{Uramoto21}; this relationship to Witt vectors cannot emerge if we see $G_{K, S}^{ab}$ individually. Our terminology above, i.e.\ deformation of a profinite group, is to suggest that this aspect of $DR_K$ may develop to yet another ``Galois deformation theory'', i.e.\ a theory of deformations of Galois groups \emph{themselves} (rather than their linear representations), whose technical base will be the theory of profinite semigroups. As the first approximation of our subject to be developed, let us display again the following fully-abstract definition:

\begin{defn}[deformation of a profinite group]
\label{full abstract deformation}
Let $G$ be a profinite group. By a \emph{deformation of $G$}, we shall mean a profinite monoid $D$ whose unit group $D^\times$ is isomorphic to $G$. 
\end{defn}

To make the subject concrete, this paper focuses on a specific deformation of the absolute Galois groups $G_K$ of number fields $K$. To be precise, we construct a deformation $\D G_K$ of $G_K$ such that (i) the maximal abelian quotient of $\D G_K$ is isomorphic to $DR_K$; (ii) $\D G_K$ is topologically generated by \emph{Frobenius elements}\footnote{It would be meaningful to mention here that these elements $\psi_v \in \D G_K$ can be defined \emph{canonically}, cf.\ \S \ref{s3}; this is in a sharp contrast to the fact that a lift $\sigma_\idp$ of the Frobenius $\frob_\idp$ onto $G_K$ can only be defined non-canonically (i.e.\ need some non-canonical choice).} $\{\psi_v\}_{v \in \V_K}$ indexed by the set $\V_K$ of finite primes $v$ of the algebraic closure $\bar{K}$ of $K$ such that each $\psi_v$ is mapped to the Frobenius $\psi_{\idp_v}$ in $DR_K$ (for the prime $\idp_v$ of $K$ corresponding to $v$) under the abelianization $\D G_K \twoheadrightarrow DR_K$; (iii) the maximal Galois groups $G_{K, S}$ with restricted ramifications $S \subseteq P_K$ all appear as the maximal closed subgroups at idempotents of $\D G_K$; (iv) all maximal closed subgroups at idempotents of $\D G_K$ are of this form; also (v) the finite $\D G_K$-sets are dual to the algebras of $\idp$-typical Witt vectors in the sense of \cite{Borger_Smit11} locally at all $\idp$ (cf.\ \S \ref{s4}). We shall firstly construct a \emph{free} deformation $DG_K$ of $G_K$ satisfying (i), (ii), (v) (cf.\ \S \ref{s3}); our target $\D G_K$ is constructed as a certain quotient\footnote{Actually, our construction of $\D G_K$ will be based on a refinement of the original definition $DR_K = \lim_\idf DR_\idf$ by a finer approximation $DR_K = \lim_{L/K} DR_{L/K}$ indexed by finite abelian extensions $L/K$ (cf.\ \S \ref{s2}); the major purpose of this refinement is that the construction of these $DR_{L/K}$'s makes equal sense for finite non-abelian extensions $L/K$. To this end, we will recall the Galois-orbit decomposition $DR_\idf = \coprod_{\idd\mid \idf} C_{\idf/\idd}$; and re-interpret the quotients of the ray class groups $C_\idf \twoheadrightarrow C_{\idf/\idd}$ (that make sense only for ray class fields $K_\idf/K$) as the quotients of $Gal(K_\idf/K)$ by its higher ramifications subgroups (that makes sense for any abelian and non-abelian extensions $L/K$).} of $DG_K$ (cf.\ \S \ref{s5}). 

This paper consists of four major sections (\S \ref{s2} -- \S \ref{s5}), with an appendix \S \ref{a1} for some generalities on semigroup theory. In \S \ref{s2} we first review the definition of the Deligne-Ribet monoid $DR_K$ and its adelic description due to Yalkinoglu \cite{Yalkinoglu} (\S \ref{s2.1}); the major subject here is to clarify the relationship between the structure of $DR_K$ and the ramification subgroups in $G_K^{ab}$ (\S \ref{s2.2}): we prove in particular that the maximal abelian Galois groups $G_{K, S}^{ab}$ with restricted ramifications appear as the maximal closed subgroups at idempotents of $DR_K$, and all the maximal closed subgroups at idempotents in $DR_K$ are of this form. After this reviews on $DR_K$, we construct in \S \ref{s3.1} a profinite monoid $DG_K$ and prove in \S \ref{s3.2} that (1) the unit group $DG_K^\times$ is isomorphic to $G_K$; and (2) the maximal abelian quotient of $DG_K$ is isomorphic to $DR_K$, where we also construct canonical (topological) generators of $DG_K$ that map onto the Frobenius elements $\psi_\idp \in DR_K$ under the abelianization $DG_K \twoheadrightarrow DR_K$. The subject of \S \ref{s4} is then to study the arithmetic aspect of $DG_K$: In \S \ref{s4.1}, we describe the category $\cl DG_K$ of finite $DG_K$-sets in terms of finite etale $K$-algebras equipped with $\Lambda_{O_v}$-ring structures for all $v \in \V_K$ in the sense of \cite{Borger_Smit11}; we show in \S \ref{s4.2} that, for each prime $\idp$, the fundamental monoid of the semi-galois category $\C_\idp$ of finite etale $\Lambda_{O_\idp}$-rings having integral models \cite{Borger_Smit11} is isomorphic to the profinite monoid $Q_\idp \ltimes_{I_\idp} G_\idp$, where $Q_\idp = W_\idp^+ \cup \{0_\idp\}$ denotes the one-point compactification of the \emph{Weil monoid} $W_\idp^+$ (i.e.\ a submonoid of the Weil group $W_\idp$ of the local field $K_\idp$, \S \ref{s3}) where $0_\idp$ behaves as the zero-element, and $I_\idp$ denotes the inertia subgroup of the absolute Galois group $G_\idp$ of the local field $K_\idp$; our $DG_K$ will be constructed so that $DG_K$ contains these profinite monoids $Q_\idp \ltimes_{I_\idp} G_\idp$ as its submonoids; and is (topologically) generated by them (see \S \ref{s3} for more detail). As the second major goal in \S \ref{s4.2}, we also prove yet another description of the fundamental monoid of $\C_\idp$ in terms of \emph{fields of norms} of $K_\idp$. In the last section \S \ref{s5}, we study a certain quotient $\D G_K$ of $DG_K$ and prove that it satisfies the above mentioned properties. 


\paragraph{Acknowledgements}
Some part of this work was announced at ALGI'25 on 2, September 2025 \cite{Uramoto_ALGI25}, and Kagoshima Algebra-Analysis-Geometry seminar on 18, February 2026 \cite{Uramoto_AAG26}. We are grateful to Hitoshi Furusawa and the organizers for recommending/inviting us to give our talks there, which motivated us to write this paper. This work is supported by JSPS KAKENHI Grant-No.22K03248. 

\paragraph{Data Availability}
This work as well as our all previous works was done purely theoretically, and no AI-tools or computer-aided methods were used throughout. This applies to our all works and writings unless otherwise explicitly stated.

\section{The Deligne-Ribet monoid $DR_K$}
\label{s2}
In this section, we review two descriptions of the Deligne-Ribet monoid $DR_K$: The first one is the original definition \cite{Deligne_Ribet}, which can be seen as a monoid extension of classical (strict) ray class groups; the second one \cite{Yalkinoglu} gives an adelic description of $DR_K$ (\S \ref{s2.1}). The main results in this section is those developed in \S \ref{s2.2}, where we study the semigroup structure of $DR_K$ in more detail. 

Throughout this paper, $K$ denotes a number field and $O_K$ the ring of integers in $K$. The set of non-zero maximal ideals of $O_K$ is denoted by $P_K$, while $I_K$ denotes the monoid of non-zero integral ideals of $K$. The completion of $K$ at $\idp \in P_K$ is denoted by $K_\idp$, and $O_\idp$ denotes its integer ring. The ring of adeles of $K$ is denoted by $\Ad_K$, while $\Ad_{K, f}$ denotes the ring of finite adeles; the ring of finite integral adeles is denoted by $\widehat{O}_K$. The Artin reciprocity map is denoted by $[-]: \Ad_K^\times \rightarrow G_K^{ab}$, which induces an isomorphism $\Ad_K^\times/ (\overline{K^\times \cdot K^{\times \circ}_\infty}) \simeq G_K^{ab}$ by class field theory (cf.\ e.g.\ \S 2.1.2 \cite{Uramoto23}). For simplicity, we shall often write as $s \in G_K^{ab}$ to mean the element $[s] \in G_K^{ab}$ for $s \in \Ad_K^\times$

\subsection{Preliminaries}
\label{s2.1}
The original definition of $DR_K$ can be found in \S2 \cite{Deligne_Ribet}. To review this, recall that, for each $\idf \in I_K$, the classical (strict) ray class group $C_\idf$ of the conductor $\idf$ was defined as the quotient $C_\idf := J_K^\idf/P_K^\idf$ of the group $J_K^\idf$ of (fractional) ideals prime to $\idf$ by the group $P_K^\idf$ of principal ideals $(a)$ with $a \equiv 1 \mod \idf$. The definition of $DR_K$ is based on a monoid extension of these ray class groups, which we call \emph{ray class monoids}: 

\begin{defn}[ray class monoid]
The \emph{ray class monoid $DR_\idf$ of the conductor $\idf \in I_K$} is defined as the quotient of the monoid $I_K$ by the following monoid congruence $\sim_\idf$: For $\ida, \idb \in I_K$, 
\begin{eqnarray}
 \ida \sim_\idf \idb &\Leftrightarrow& \exists t \in K_+ \cap (1 + \idf \idb^{-1}), \hspace{0.1cm} \ida \idb^{-1} = (t), 
\end{eqnarray}
where $K_+$ denotes the totally positive elements of $K$. We denote by $[\ida]_\idf$, or simply by $[\ida] \in DR_\idf$, to mean the $\sim_\idf$-class of $\ida \in I_K$. 
\end{defn}

By definition, for $\idf \mid \idf'$, we have the implication $\ida \sim_{\idf'} \idb \Rightarrow \ida \sim_\idf \idb$ for all $\ida, \idb \in I_K$, which means that there exists a canonical surjective (monoid) homomorphism $\rho^{\idf'}_\idf: DR_{\idf'} \twoheadrightarrow DR_\idf$. In fact, these homomorphisms constitute an inverse system $\{DR_\idf, \rho_\idf^{\idf'}\}_\idf$ of finite monoids, whose inverse limit is the Deligne-Ribet monoid $DR_K$:
\begin{defn}[Deligne-Ribet monoid $DR_K$]
The \emph{Deligne-Ribet monoid} $DR_K$ for the number field $K$ is defined as the following inverse limit of the ray class monoids $DR_\idf$:
\begin{eqnarray}
 DR_K &:=& \lim_{\leftarrow \idf} DR_\idf. 
\end{eqnarray}
\end{defn}

\begin{rem}[the unit group of $DR_\idf$]
To gain some intuition on the structure of $DR_\idf$, it is helpful to recall that the congruence $\ida \sim_\idf \idb$ is equivalent to the following condition:
\begin{eqnarray}
 \ida \sim_\idf \idb &\Leftrightarrow& \left\{ \begin{array}{l}  \textrm{we have $(\ida, \idf) = (\idb, \idf) =: \idd$ and} \\  \ida/\idd = \idb/\idd \hspace{0.2cm} \textrm{in $C_{\idf/\idd}$} \end{array} \right.
\end{eqnarray}
where $(\ida, \idf)$ (resp.\ $(\idb, \idf)$) denotes the greatest common divisor of $\ida$ (resp.\ $\idb$) and $\idf \in I_K$. Therefore, it follows in particular that, for $\ida, \idb$ prime to $\idf$, the congruence $\ida \sim_\idf \idb$ is equivalent to the equality $\ida = \idb$ in the ray class group $C_\idf$; and hence, the unit group $DR_\idf^\times$ can be naturally identified with the group $C_\idf$ via a well-defined injection $C_\idf \hookrightarrow DR_\idf$. Taking the inverse limit of $DR_\idf$'s, it also follows from class field theory that the unit group $DR_K^\times = \lim_\idf DR_\idf^\times$ of $DR_K$ is canonically isomorphic to the maximal abelian Galois group $G_K^{ab}$.
\end{rem}

\begin{rem}[Galois-orbit decomposition of $DR_\idf$]
Thanks to the above, the ray class monoid $DR_\idf$ has canonically a structure of a \emph{finite $G_K$-set}, namely, with respect to the (right) action of $G_K$ via $G_K \twoheadrightarrow G_K^{ab} \twoheadrightarrow C_\idf = DR_\idf^\times \hookrightarrow DR_\idf$. With this Galois action, $DR_\idf$ admits the following $G_K$-orbit decomposition:
\begin{eqnarray}
 DR_\idf &=& \coprod_{\idd \mid \idf} [\idd] \cdot DR_\idf^\times. 
\end{eqnarray}
Each class $[\idd]$, moreover, has generally a non-trivial isotropy group so that the $G_K$-orbit $[\idd] \cdot DR_\idf^\times$ is isomorphic to $C_{\idf/\idd}$ as finite $G_K$-sets, where the $G_K$-set structure of $C_{\idf/\idd}$ comes from the canonical projection $G_K \twoheadrightarrow C_{\idf/\idd}$. In this sense, $DR_\idf$ can be decomposed as the following disjoint sum of the ray class groups $C_{\idf/\idd}$ in the category $\cl G_K$ of finite $G_K$-sets:
\begin{eqnarray}
\label{galois orbit of DR_f}
 DR_\idf &\simeq& \coprod_{\idd \mid \idf} C_{\idf/\idd}
\end{eqnarray}
Rephrased in terms of semigroup theory, this decomposition coincides with that of $DR_\idf$ by Green's \emph{$\HH$-equivalence classes} (cf.\ \S \ref{a1}). We will further study this semigroup-theoretic aspect in \S \ref{s2.2}. 
\end{rem}

\begin{rem}[adelic description of $DR_K$]
The Deligne-Ribet monoid $DR_K$ also has the following description due to Yalkinoglu (cf.\ Proposition 8.2 \cite{Yalkinoglu}):
\begin{eqnarray}
\label{partly adelic DR_K}
DR_K &\simeq& \widehat{O}_K \times_{\widehat{O}_K^\times} G_K^{ab},
\end{eqnarray}
where the group $\widehat{O}_K^\times$ acts on the product monoid $\widehat{O}_K \times G_K^{ab}$ by $u \cdot (\rho, s) = (\rho u, u^{-1} s)$ for each $\rho \in \widehat{O}_K$, $[s] \in G_K^{ab}$ and $u \in \widehat{O}_K^\times$; the right-hand side of (\ref{partly adelic DR_K}) is the quotient of $\widehat{O}_K \times G_K^{ab}$ by this action of $\widehat{O}_K^\times$. Together with the isomorphism $G_K^{ab} \simeq \Ad_K^\times/(\overline{K^\times \cdot K^{\times, \circ}_\infty})$ of class field theory, the above description yields a description of $DR_K$ purely in terms of adeles:
\begin{eqnarray}
\label{adelic DR_K}
 DR_K &\simeq&  \widehat{O}_K \times_{\widehat{O}_K^\times} \bigl(\Ad_K^\times/(\overline{K^\times \cdot K^{\times, \circ}_\infty})\bigr)
\end{eqnarray}
In this paper, we will utilize the former adelic description (\ref{partly adelic DR_K}) of $DR_K$ (or its finite approximation) in \S \ref{s2.2} in order to refine the original approximation $DR_K = \lim_\idf DR_\idf$, based on which we construct our non-abelian profinite monoid $\D G_K$ in \S \ref{s5}. 
\end{rem}

\begin{rem}[$DR_K$ and $\Lambda$-rings \cite{Borger_Smit08, Borger_Smit11}]
The Deligne-Ribet monoid $DR_K$ has a notable relationship with class field theory of the number field $K$ as proved by Borger and de Smit \cite{Borger_Smit08, Borger_Smit11} (see also \cite{Borger_Smit18}), where they proved that the category $\cl DR_K$ of finite $DR_K$-sets is (opposite) equivalent to that $\C_K$ of \emph{finite etale $\Lambda$-rings over $K$ having integral models} (cf.\ e.g.\ \S 2 \cite{Uramoto21}). In particular, each $\idp \in I_K$ canonically embeds in $DR_K$; and we denote its image by $\psi_\idp \in DR_K$ and call the \emph{Frobenius element}. This notation comes from the fact that $\psi_\idp \in DR_K$ corresponds to the Frobenius lifts on each finite etale $\Lambda$-rings in $\C_K$. 
\end{rem}

\subsection{The structure of $DR_K$}
\label{s2.2}
In this subsection, we study the semigroup structure of $DR_K$ in more detail; in particular, we show that the maximal abelian Galois groups $G_{K, S}^{ab}$ with restricted ramifications $S \subseteq P_K$ all appear as the maximal closed subgroups at idempotents of $DR_K$ (\S \ref{s2.2.1}). In some sense, this structure of $DR_K$ around idempotents is a limiting result of the fact that the finite quotient $DR_\idf$ reflects the filtration of higher ramification subgroups of $Gal(K_\idf/K)$. For our later purpose, we shall highlight this higher ramification-theoretic aspect of $DR_\idf$'s and extend them to finite monoids $DR_{L/K}$ indexed by any finite abelian extensions $L/K$ (\S \ref{s2.2.2}), so that it makes sense also for finite non-abelian extensions $L/K$ (cf.\ \S \ref{s5}). For some terminologies on semigroup theory, the reader is referred to \S \ref{a1}, or \cite{Pin, Howie}.

\subsubsection{The local structure of $DR_K$}
\label{s2.2.1}
Here we study the \emph{local structure} of $DR_K$, i.e.\ the structure of \emph{Green equivalence classes} of $DR_K$; since the $\JJ, \RR, \LL$, $\HH$-classes coincide for $DR_K$ by its commutativity, we just describe $\HH$-classes. In general, $\HH$-classes are classified into two types, i.e.\ regular and non-regular $\HH$-classes. By definition, the former ones are those $\HH$-classes containing idempotents and form the maximal closed subgroups at their respective idempotents, whereas the latter ones are not even closed under multiplication. In the case of $DR_K$, both types of $\HH$-classes are relatively simple and admit complete descriptions as we see in this subsection. 

In fact, we first remark that the $\HH$-classes of $DR_K$ are precisely the $G_K$-orbits:

\begin{prop}[$\HH$-classes are $G_K$-orbits]
Two elements $x, y \in DR_K$ are $\HH$-equivalent in $DR_K$ if and only if they belongs to the same $G_K$-orbit. 
\end{prop}
\begin{proof}
This is essentially proved in \S 3.2, \cite{Uramoto21}. The only slight difference is that we stated it in terms of $\JJ$-classes rather than $\HH$-classes; but as mentioned above, $\JJ$-classes and $\HH$-classes are the same for the commutative $DR_K$, hence the claim. 
\end{proof}

Therefore, the classification of the $\HH$-classes of $DR_K$ is equivalent to that of $G_K$-orbits in $DR_K$; and as mentioned above, they are further classified into the regular and non-regular $\HH$-classes. For our classification of $\HH$-classes, we start with the regular ones. By the general theory of semigroups, a regular $\HH$-class $H \subseteq DR_K$ is that of some idempotent $e \in DR_K$ and forms a sub-semigroup of $DR_K$ which itself forms a group with $e$ the identity element. In particular, we have the description $H = (e DR_K e)^\times = (e DR_K)^\times$, where the last equality is due to the commutativity of $DR_K$. On the other hand, as proved in our previous work \cite{Uramoto25}, each idempotent $e \in DR_K$ is of the form $e = e_S$ for some unique $S \subseteq P_K$, where the idempotent $e_S \in DR_K$ is given explicitly as follows:
\begin{eqnarray}
 e_S &=& [1_S, 1];
\end{eqnarray}
here the $\idp$-component of $1_S \in \widehat{O}_K = \prod_\idp O_\idp$ is given as $1 \in O_\idp$ if $\idp \in S$ and $0 \in O_\idp$ otherwise. Using these results, we obtain the following description of regular $\HH$-classes:

\begin{prop}[classification of regular $\HH$-classes]
\label{classification of regular H class}
Each regular $\HH$-class of $DR_K$ is that of the idempotent $e_S$ for some unique $S \subseteq P_K$; and for each $S \subseteq P_K$, the $\HH$-class $H_{e_S}$ is isomorphic (as a profinite group) to the maximal abelian Galois group $G_{K, S}^{ab}$ with the ramifications restricted on $S$. 
\end{prop}
\begin{proof}
We need to prove the isomorphism $(e_S DR_K)^\times \simeq G_{K, S}^{ab}$. By the adelic description of $DR_K$, we have:
\begin{eqnarray}
 e_S \cdot DR_K &=& [1_S, 1] \cdot \widehat{O}_K \times_{\widehat{O}_K^\times} G_K^{ab}.
\end{eqnarray}
In particular, since the $\idp$-components of $1_S$ are zero for $\idp \not \in S$, the $\widehat{O}_K$-component of $DR_K$ reduces to $\prod_{\idp \in S} O_\idp$; accordingly, the $\idp$-components of the unit group $\widehat{O}_K^\times$ get trivial for $\idp \not \in S$. This transfers to the group component (via the congruence $(\rho u, \sigma) \sim (\rho, u \sigma)$) so that the unit group of $e_S \cdot DR_K$ is isomorphic to the quotient of $G_K^{ab}$ by the product of the groups $O_\idp^\times$ for $\idp \not \in S$, and thus, isomorphic to $G_{K, S}^{ab}$ by class field theory. 
\end{proof}

Now we proceed to the classification of non-regular $\HH$-classes; as proved above, this is equivalent to classifying those $G_K$-orbits which contain no idempotents. With this in mind, we start with the following easy observations, one of which generalizes Proposition \ref{classification of regular H class} to non-regular $\HH$-classes:

\begin{lem}
\label{H-equivalence and valuation}
Two elements $x = [\rho, \sigma]$, $y = [\lambda, \tau] \in DR_K$ belong to the same $G_K$-orbit if and only if $v_\idp(\rho) = v_\idp(\lambda)$ for all $\idp \in P_K$.
\end{lem}
\begin{proof}
The if part is easy; to see the only-if part, suppose that $x=[\rho, \sigma]$ and $y = [\lambda, \tau]$ satisfy the condition $v_\idp(\rho) = v_\idp(\lambda)$ for all $\idp$. Then, for each $\idp$, there exists a unit $u_\idp \in O_\idp^\times$ such that $\rho_\idp = \lambda_\idp u_\idp$; thus, letting $u=(u_\idp)$, we have $\rho = \lambda u$ in $\widehat{O}_K$. Therefore, $[\lambda, \tau] = [\rho u, \tau] = [\rho, u \tau]$, which clearly belongs to the same $G_K$-orbit as $[\rho, \sigma]$. This completes the proof. 
\end{proof}

\begin{lem}
The $G_K$-orbit of $x = [\rho, \sigma]$ contains an idempotent if and only if $(v_\idp(\rho)) \in \prod_\idp \{0, \infty\}$.
\end{lem}
\begin{proof}
This follows from the fact that $\prod_\idp \{0, \infty\}$ is equal to the set of the values $(v_\idp (1_S))$ where $S$ ranges over all subsets of $P_K$. 
\end{proof}

\begin{prop}[classification of general $\HH$-classes]
\label{classification of general H class}
The $\HH$-class of $x = [\rho, \sigma]$ is isomorphic (as a profinite $G_K$-set) to the maximal abelian Galois group $G_{K, S}^{ab}$ with $S = \{\idp \in P_K \mid v_\idp(x) \neq \infty\}$. 
\end{prop}
\begin{proof}
The proof is essentially the same as that of Proposition \ref{classification of regular H class}. The slight difference is only the fact that non-regular $\HH$-classes are not closed under multiplication; thus the isomorphism is not as a profinite group but only as a profinite $G_K$-set. 
\end{proof}

As seen above, the $\HH$-classes are the $G_K$-orbits of $DR_K$; but by the monoid structure of $DR_K$, they constitute a poset whose adjacent relation is described by the action of non-invertible Frobenius elements $\psi_\idp \in DR_K$. To see this aspect, firstly let $x, y \in DR_K$ be two elements; we write as $x \leq y$\footnote{This relation is equivalent to the conventional quasi-ordering $x \leq_\JJ y$ (and also, $x \leq_\RR$ and $x \leq_\LL y$); see \S \ref{a1}. We omit the subscript because of the commutativity of $DR_K$.} if there exists $z \in DR_K$ such that $x = y z$; then, we can see that $x$ and $y$ are $\HH$-equivalent if and only if $x \leq y$ and $y \leq x$. Therefore, it makes sense to define $H_1 \leq H_2$ for two $\HH$-classes $H_1, H_2 \subseteq DR_K$ when $x_1 \leq x_2$ for their representatives $x_1 \in H_1$ and $x_2 \in H_2$. In this sense, the $\HH$-classes constitute a poset. 

To describe this poset structure, we further denote as $H_1 \triangleleft H_2$ when (1) $H_1 \leq H_2$, (2) $H_1 \neq H_2$, and (3) $H_1 \leq H' \leq H_2$ implies $H'=H_1$ or $H' = H_2$. Then we have the following characterization of this relation $\triangleleft$ in terms of the Frobenius elements $\psi_\idp$:

\begin{prop}[adjacent relation $\triangleleft$ of $\HH$-classes]
For two $\HH$-classes $H_1\neq H_2$, we have $H_1 \triangleleft H_2$ if and only if there exists some $\idp \in P_K$ such that $H_1 = \psi_\idp H_2$.
\end{prop}
\begin{proof}
By Lemma \ref{H-equivalence and valuation}, the poset of $\HH$-classes is isomorphic to $\prod_\idp (\nat \cup \{\infty\})$ whose poset structure is defined naturally with component-wise (natural) orders. Thus, for two $\HH$-classes $H_1\neq H_2$ which correspond to $v_1 = (v_{1, \idp}), v_2 = (v_{2, \idp}) \in \prod_\idp (\nat \cup \{0\})$, we see that $H_1 \triangleleft H_2$ if and only if $v_1 = v_2 + e_\idp$ for some $\idp \in P_K$ with $v_{2, \idp} \neq \infty$, where $e_\idp \in \prod_\idp (\nat \cup \{0\})$ is $1$ at $\idp$, while $0$ at $\idq \neq \idp$. Recalling that $\psi_\idp = [\pi_\idp, \pi_\idp^{-1}]$, the latter implies the equality $H_1 = \psi_\idp H_2$ for such a $\idp$. Conversely, if $H_1 \neq H_2$ and $H_1 = \psi_\idp H_2$, then $H_2 \neq \psi_\idp H_2$, whence $v_{2, \idp} \neq \infty$. Again, by $\psi_\idp = [\pi_\idp, \pi_\idp^{-1}]$, we have $v_1 = v_2 + e_\idp$ for the corresponding $v_1, v_2$, which implies the desired adjacent relation $H_1 \triangleleft H_2$. This completes the proof. 
\end{proof}

In general, we shall write as $H_v \subseteq DR_K$ for $v \in \prod_\idp (\nat \cup \{\infty\})$ the $\HH$-class of $[\rho_v, 1] \in DR_K$ such that $v_\idp(\rho_v) = v_\idp$; then, we have $H_{v_1} \leq H_{v_2}$ if and only if $v_2 \leq v_1$. Conversely, we shall write as $v_H$ for the element of $\prod_\idp (\nat \cup \{\infty\})$ that corresponds to the $\HH$-class $H$. With this notation, we define $S_H \subseteq P_K$ for each $\HH$-class $H$ by $S_H := \{\idp \in P_K \mid (v_H)_\idp \neq \infty \}$, which we shall call the \emph{support (of ramifications)} of $H$ (cf.\ Proposition \ref{stability} below).

As observed above, the action of the Frobenius elements $\psi_\idp$ characterizes the adjacent relation $H_1 \triangleleft H_2$ for \emph{distinct} $\HH$-classes $H_1 \neq H_2$; in general, however, their actions do not necessarily change the $\HH$-classes. In fact, the stability of an $\HH$-class $H$ under the Frobenius action $\psi_\idp$ is characterized by the infiniteness of the valuation $(v_H)_\idp = \infty$ at $\idp$, which is precisely when $\idp$ is unramified on this $G_K$-orbit component $H$: 

\begin{prop}[$\psi_\idp$-stability of an $\HH$-class]
\label{stability}
For an $\HH$-class $H \subseteq P_K$, we have $H = \psi_\idp H$ if and only if $\idp \not \in S_H$; in particular, when this is the case, the action of $\psi_\idp$ on $H$ coincides with that of the Frobenius $\sigma_\idp \in G_{K, S_H}^{ab}$. 
\end{prop}
\begin{proof}
By rephrasing the first claim in terms of $v_H \in \prod_\idp (\nat \cup \{\infty\})$, we can see that $H = \psi_\idp H$ if and only if $(v_H)_\idp = \infty$, hence the first claim. For the second claim, notice that $H = e_{S_H} H$, where $e_{S_H} = [1_{S_H}, 1]$ is the idempotent corresponding to $S_H$; by $\idp \not \in S_H$, we deduce $\psi_\idp H = [\pi_\idp, \pi_\idp^{-1}] H = [\pi_\idp, \pi_\idp^{-1}] \cdot [1_{S_H}, 1] H = [1, \pi_\idp^{-1}] [1_{S_H}, 1] H = \sigma_\idp^{-1} H$, hence the claim. (Here recall that $\sigma \in G_K$ acts on $DR_K = \widehat{O}_K \times_{\widehat{O}_K^\times} G_K^{ab}$ by $\sigma \mapsto [1, \sigma^{-1}]$; see e.g.\ Proposition 8.2 \cite{Yalkinoglu}.)
\end{proof}

These propositions describe how the Frobenius elements $\psi_\idp$ act on the poset of the $\HH$-classes of $DR_K$; in particular, note that a single action of $\psi_\idp$ preserves the isomorphism class of profinite $G_K$-sets: if $H_1 = \psi_\idp H_2$, we have $S_{H_1} = S_{H_2} =: S$, thus, the $\HH$-classes $H_1, H_2 \subseteq DR_K$ are isomorphic as profinite $G_K$-sets by the $G_K$-equivariant map $\psi_\idp: H_2 \rightarrow H_1$ within $DR_K$ (in fact, both isomorphic to $G_{K, S}^{ab}$). 

In general, however, a \emph{degeneration} of $\HH$-classes may occur when we take \emph{limits} of the Frobenius actions: As the simplest case, let $H \subseteq DR_K$ be an $\HH$-class with $\emptyset \neq S_H =: S \subseteq P_K$; and let $\idp \in S$. As seen above, the $\HH$-classes $\psi_\idp^n H$ have the same support $S \subseteq P_K$ for any $n \in \nat$, so that they are all isomorphic to $G_{K, S}^{ab}$ as $G_K$-sets; but the $\HH$-class $\psi_\idp^\omega H$, where $\psi_\idp^\omega = \lim_n \psi_\idp^{n!}$, has a smaller support $S' := S \setminus \{\idp\}$ so that it degenerate as a $G_K$-set to the smaller Galois group $G_{K, S'}^{ab}$. More generally, the idempotents $e_S$ of $DR_K$ ($S \subseteq P_K$) are all represented as suitable product of such idempotents $\psi_\idp^{\omega}$, that is, $e_S = \prod_{\idp \not \in S} \psi_\idp^\omega$; and the actions of these idempotents $e_S = \prod_{\idp \not \in S} \psi_\idp^{\omega}$ induce degenerations of the group structure of $G_{K, S_H}^{ab}$ for each $\HH$-class $H$: to be more specific, if $H' = e_S \cdot H$, then $H'$ has the support $S_{H'} = S_H \setminus S$ of ramifications (smaller than $S_H$ in general); thus, the $\HH$-class $e_S \cdot H$ is isomorphic to the (generally smaller) Galois group $G_{K, S_{H'}}^{ab}$ as a $G_K$-set. In this way, the $\HH$-classes of $DR_K$ may degenerate in $DR_K$ by the limits of the actions of Frobenius elements $\psi_\idp$.

\subsubsection{The local structure of $DR_\idf$ and their refinement}
\label{s2.2.2}
To obtain a better understanding of this poset structure of $\HH$-classes in $DR_K$, it is helpful to recall the similar structure of its finite quotients $DR_\idf$ (cf.\ \S \ref{s2.1}). In fact, the above ramification-theoretic structure of the poset of $\HH$-classes in $DR_K$ can be understood more finely as the limit of the posets of $\HH$-classes of $DR_\idf$'s: As we highlight below, the posets of $\HH$-classes of $DR_\idf$'s reflect the structure of filtration of \emph{higher ramification subgroups} of $Gal(K_\idf/K)$. 

To this end, we begin with recalling yet another description of $DR_\idf$'s \cite{Yalkinoglu}, and extend it firstly to finite monoids $DR_{L/K}$ indexed by finite abelian extensions $L/K$, which will also converge to $DR_K$, that is, $DR_K = \lim_{L/K} DR_{L/K}$. By Proposition 8.2 \cite{Yalkinoglu}, $DR_\idf$ has the following description for each conductor $\idf \in I_K$:
\begin{eqnarray}
\label{adelic DR_f}
 DR_\idf &\simeq& (O_K/\idf) \times_{(O_K/\idf)^\times} C_\idf.
\end{eqnarray}
The adelic description of $DR_K$ is given by taking the inverse limit of (\ref{adelic DR_f}). In this description, we can replace $C_\idf$ with the (isomorphic) Galois group $Gal(K_\idf/K)$ so that $(O_K/\idf) \times_{(O_K/\idf)^\times} Gal(K_\idf/K)$ makes sense for any abelian Galois groups $Gal(L/K)$. Indeed, let $L/K$ be a finite abelian extension of $K$ and $\idf_{L/K} \in I_K$ be the conductor of $L/K$. Then, similarly to the right-hand side of (\ref{adelic DR_f}), we define a finite monoid $DR_{L/K}$ as follows:
\begin{eqnarray}
 DR_{L/K} &:=& (O_K/ \idf_{L/K}) \times_{(O_K/\idf_{L/K})^\times} Gal(L/K)
\end{eqnarray}
where the group $(O_K/\idf_{L/K})^\times$ acts on the monoid $(O_K/\idf_{L/K}) \times Gal(L/K)$ as $u \cdot (\rho, \sigma) = (\rho u, [u]^{-1} \sigma)$ for $u \in (O_K/\idf_{L/K})^\times = \prod_\idp (O_\idp/\idp^{v_\idp (\idf_{L/K})})^\times$, $\rho \in O_K/\idf_{L/K}$ and $\sigma \in Gal(L/K)$. Note that this action of $u$ is well-defined since $u \equiv u' \mod \idf_{L/K}$ if and only if $u = \lambda u'$ for some $\lambda \in \prod_\idp U_\idp^{(v_\idp (\idf_{L/K}))}$ where $U_\idp^{(v_\idp(\idf_{L/K}))}$ denotes the group of $v_\idp(\idf_{L/K})$-th local units which vanishes in $Gal(L/K)$ by definition of the conductor $\idf_{L/K}$. Since the conductor of the ray class field $K_\idf$ is not generally equal to $\idf$, the new  monoids $DR_{L/K}$ do not really ``refine'' $DR_\idf$'s in that $DR_{K_\idf/K}$ is not actually the same as $DR_\idf$; but $DR_{L/K}$ satisfies the same ramification-theoretic properties as $DR_\idf$, and also approximates $DR_K$ in the following sense:

\begin{prop}
The above finite monoids $DR_{L/K}$ satisfy the following properties: 
\begin{enumerate}
\item $DR_{L/K}^\times = Gal(L/K)$; 
\item there exists a surjective homomorphism $I_K \twoheadrightarrow DR_{L/K}$ such that $\idp \in P_K$ is unramified in $L/K$ if and only if $\idp$ is invertible in $DR_{L/K}$; 
\item there exists a surjective homomorphism $DR_{F/K} \twoheadrightarrow DR_{L/K}$ for finite extensions $F \supseteq L \supseteq K$ such that $\{DR_{L/K}\}$ constitute an inverse system whose limit is isomorphic to $DR_K$.
\end{enumerate}
\end{prop}
\begin{proof}
By definition, it is not hard to see that, for any $L/K$, there exists a surjective homomorphism $DR_\idf \twoheadrightarrow DR_{L/K}$ with $\idf = \idf_{L/K}$; combined with the quotient $I_K \twoheadrightarrow DR_\idf$, we obtain $I_K \twoheadrightarrow DR_{L/K}$. To complete the proof of the second claim, let $\pi_\idp$ be a uniformizer of $\idp \in P_K$; then $\idp = [\pi_\idp, \pi_\idp^{-1}]$ in $DR_K = \widehat{O}_K \times_{\widehat{O}_K^\times} G_K^{ab}$. Note that $[\pi_\idp, \pi_\idp^{-1}]$ is invertible in the finite quotient $DR_{L/K}$ if and only if its first component $\pi_\idp \in \widehat{O}_K$ gets invertible in the multiplicative monoid $O_K/\idf_{L/K}$, namely, if and only if $\idp$ does not divide $\idf_{L/K}$, hence the second claim; the first claim follows similarly. Concerning the third claim, its first half is clear by definition; to see the second half, notice that we have not only the quotient $DR_{\idf_{L/K}} \twoheadrightarrow DR_{L/K}$ for each $L/K$ but conversely $DR_{L/K} \twoheadrightarrow DR_\idf$ for each $\idf \in I_K$ as well with some finite abelian extension $L/K$. (Just take $L/K$ so that $\idf \mid \idf_{L/K}$ and $K_\idf \subseteq L$.) This completes the proof. 
\end{proof}

\begin{prop}[Galois-orbit decomposition of $DR_{L/K}$]
Let $L/K$ be a finite abelian extension and $\idf = \idf_{L/K}$ be the conductor of $L/K$. Then the finite monoid $DR_{L/K}$ has the following $G_K$-orbit decomposition in the category $\cl G_K$ of finite $G_K$-sets:
\begin{eqnarray}
\label{decomposition of D_LK}
 DR_{L/K} &=& \coprod_{\idd \mid \idf} Gal(L^{(\idf/\idd)}/K)
\end{eqnarray}
where $L^{(\idf/\idd)}$ denotes the intermediate field of $L/K$ corresponding to the quotient of $Gal(L/K)$ by the composition of the ramification subgroups $Gal^{e_\idp}(L_\idp/K_\idp)$ in the upper numbering for $e_\idp = v_\idp(\idf/\idd)$. 
\end{prop}
\begin{proof}
Throughout this proof (and the above statement), let $L_\idp$ denote the local field completing $L$ at any prime $\idP \in P_L$ over $\idp \in P_K$; since $L/K$ is now abelian, the subgroups $Gal^e(L_\idp/K_\idp)$ does not depend on the choice of $\idP \mid \idp$. To show the claim (\ref{decomposition of D_LK}), recall first that the ramificaction subgroup $Gal^e(L_\idp/K_\idp)$ corresponds to the group $U_\idp^{(e)} = 1 + \idp^e O_\idp \subseteq O_\idp^\times$ of the $e$-th local units (Theorem 6.2 \cite{Neukirch}); on the other hand, by definition, $DR_{L/K}$ has the following $G_K$-orbit decomposition:
\begin{eqnarray}
 DR_{L/K} &=& \coprod_{\idd \mid \idf} [\rho_\idd, 1] \cdot DR_{L/K}^\times,
\end{eqnarray}
where $\rho_\idd \in \widehat{O}_K$ is taken as $\rho_\idd \widehat{O}_K \cap K = \idd$. Note also that the isotropy group of $\rho_\idd \in O_K/\idf$ for the action of $\widehat{O}_K^\times$ is the group of units $\lambda = (\lambda_\idp) \in \widehat{O}_K^\times$ with $\lambda_\idp \equiv 1 \mod \idp^{v_\idp(\idf/\idd)}$, that is, $\lambda_\idp \in U_\idp^{(e_\idp)}$.  It follows from these facts that the Galois-orbit $[\rho_\idd, 1] \cdot DR_{L/K}^\times$ is isomorphic as a right $G_K$-set to the Galois group $Gal(L^{(\idf/\idd)}/K)$, hence the claim. 
\end{proof}

\begin{rem}[$\HH$-classes are $G_K$-orbits]
As in the case of $DR_K$, the $\HH$-classes of $DR_{L/K}$ are the $G_K$-orbits. Therefore, the above decomposition $DR_{L/K} = \coprod_{\idd \mid \idf} Gal(L^{(\idf/\idd)}/K)$ coincides with that by $\HH$-classes. Moreover, projecting the Frobenius element $\psi_\idp \in DR_K$ onto $DR_{L/K}$, we can consider the action of Frobenius elements $\psi_\idp$ on the $\HH$-classes of $DR_{L/K}$; however, while the action of $\psi_\idp$ in the case of $DR_K$ preserves the structure of $\HH$-classes $H$ as $G_K$-sets for $\idp \in S_H$, the action of $\psi_\idp$ on the $\HH$-classes of $DR_{L/K}$ almost always\footnote{More precisely, the degeneration can occur when the upper numbering of ramification subgroups is at the jump.} degenerates their structures as finite $G_K$-sets. In fact, the $\HH$-class $Gal(L^{(\idf/\idd)}/K) \subseteq DR_{L/K}$ degenerates to the smaller $\HH$-class $Gal(L^{(\idf/\idp\idd)}/K)$ when $\idp\idd \mid \idf$ by the action of $\psi_\idp$, which is in a sharp contrast to the corresponding fact for $DR_K$ that the $\HH$-class $H \subseteq DR_K$ maps \emph{isomorphically} as $G_K$-sets to the other $\HH$-class $\psi_\idp H$ when $\idp \in S_H$ by the action of $\psi_\idp$. In fact, the latter isomorphism $\psi_\idp: H \rightarrow \psi_\idp H$ of $\HH$-classes $H \subseteq DR_K$ should be understood more finely as the result of ``leveling'' the degenerations $\psi_\idp: Gal(L^{(\idf/\idd)}/K) \rightarrow Gal(L^{(\idf/\idp\idd)}/K)$ by taking the limit of larger $L/K$ (thus larger $\idf \in I_K$). 
\end{rem}

\begin{rem}[further re-interpretation of $DR_{L/K}$]
\label{nonabelian DR_f}
Here we give yet another description of $DR_{L/K}$ so that it makes sense for all finite \emph{non-abelian} Galois extensions $L/K$ too. As proved above, $DR_{L/K}$ could be described in terms of higher ramification subgroups of $Gal(L/K)$ as finite $G_K$-sets: That is,
\begin{eqnarray}
DR_{L/K} &=& \coprod_{\idd \mid \idf} Gal(L^{(\idf/\idd)}/K)
\end{eqnarray}
where $L^{(\idf/\idd)}$ denotes the fixed field of the composition of the ramification subgroups $Gal^{e_\idp}(L_\idP/K_\idp)$ with $e_\idp = v_\idp(\idf/\idd)$ choosing one $\idP$ over $\idp$; for a non-abelian Galois extension $L/K$, we take these ramification subgroups for all $\idP$ over $\idp$, so that their compositions are \emph{normal} subgroups of $Gal(L/K)$, whence $Gal(L^{(\idf/\idd)}/K)$ makes sense. 

Further we shall rewrite this description isomorphically as:
\begin{eqnarray}
\label{yet another}
DR_{L/K} &=& \coprod_{\idd \mid \idf} \{\idd\} \times Gal(L^{(\idf/\idd)}/K)
\end{eqnarray}
where the additional first component of $\{\idd\} \times Gal(L^{(\idf/\idd)}/K)$ is to remember the data $\idd \mid \idf$ and use it to define a monoid structure on the finite set of the right-hand side of (\ref{yet another}) as follows:
\begin{eqnarray}
 (\idd, \sigma) \cdot (\idd', \sigma') &:=& (\idd'', \sigma \cdot \sigma')
\end{eqnarray}
where $\idd'' := (\idf, \idd \idd')$ and the product $\sigma \cdot \sigma'$ of $\sigma \in Gal(L^{(\idf/\idd)}/K)$, $\sigma' \in Gal(L^{(\idf/\idd')}/K)$ is taken in the smaller group $Gal(L^{(\idf/\idd'')}/K)$, thus well-defined. It is not difficult to prove that this multiplication defines a monoid structure on the finite set of the right-hand side of (\ref{yet another}) isomorphic to $DR_{L/K}$ as finite monoids (i.e.\ more than finite $G_K$-sets). Also, since the conductor $\idf$ too can be defined in terms of higher ramification subgroups in $Gal(L/K)$ (as we discuss in more detail in \S \ref{s3.2}), this yet another description of $DR_{L/K}$ makes sense for finite non-abelian extensions $L/K$, whose structure we will study in \S \ref{s5}. 
\end{rem}

\section{Free deformation of $G_K$}
\label{s3}
Therefore, we can define a non-abelian extension of $DR_K$ by taking the inverse limit of the above  finite monoids for all (not necessarily abelian) finite extensions $L/K$ as we will do in \S \ref{s5}; but before that, we develop here one more construction of a non-abelian extension $DG_K$ of $DR_K$ by extending its \emph{adelic} description discussed in \S \ref{s2}. This $DG_K$ is characterized by the feature that its relationship to Witt vectors is canonical as we show in \S \ref{s4}, yet the above non-abelian extensions of $DR_\idf$'s are all realized as finite quotients of this $DG_K$. Further, we shall see that $DG_K$ has some useful canonical topological generators $\psi_v \in DG_K$ indexed by finite primes of the algebraic closure $\bar{K}$ of $K$, which we will call the \emph{Frobenius elements} of $DG_K$; the existence of such canonical elements is not evident from the above heuristic definition of the finite monoids in Remark \ref{nonabelian DR_f}. 

\subsection{Construction of $DG_K$}
\label{s3.1}
To sketch our construction of $DG_K$, first recall the following adelic description of $DR_K$ (cf.\ \S \ref{s2.2}):
\begin{eqnarray}
\label{adelic}
 DR_K &:=& \widehat{O}_K \times_{\widehat{O}_K^\times} G_K^{ab}. 
\end{eqnarray}
In this description, note that the first component $\widehat{O}_K$ in the right-hand side of (\ref{adelic}) is the product of the multiplicative monoids $O_\idp$ of integers of the local fields $K_\idp$; and that, as shown by Hoshi \cite{Hoshi}, the multiplicative monoid $O_\idp$ can be further described in terms of the subgroup $G_\idp \hookrightarrow G_K$, where $G_\idp := Gal(\bar{K}_\idp/K_\idp)$ denotes the absolute Galois group of the local field $K_\idp$: Let $O_\idp^* := O_\idp \setminus \{0\}$, and $I_\idp \subseteq G_\idp$ the inertia subgroup. The quotient group $G_\idp/I_\idp$ is an abelian group topologically generated by the Frobenius $\frob_\idp$; thus, there exists a surjection $q: G_\idp^{ab} \twoheadrightarrow G_\idp/I_\idp$. Noting the isomorphism $G_\idp^{ab} \simeq \hat{K}_\idp^\times$ of local class field theory, we then have a natural isomorphism of multiplicative monoids:
\begin{eqnarray}
\label{pullback}
 O_\idp^* &\simeq& G_\idp^{ab} \times_q \frob_\idp^\nat,
\end{eqnarray}
where the right-hand side is the pullback of the submonoid $\frob_\idp^\nat := \{\frob_\idp^n \mid n \in \nat\} \leq G_\idp/I_\idp$ along $q: G_K^{ab} \twoheadrightarrow G_\idp/I_\idp$. Hence the multiplicative monoid $O_\idp = O_\idp^* \cup \{0\}$ of local integers is isomorphic to the one-point compactification $\bigl( G_\idp^{ab} \times_q \frob_\idp^\nat \bigr) \cup \{0_\idp\}$ of the abstract monoid $G_\idp^{ab} \times_q \frob_\idp^\nat$, where $0_\idp$ behaves as the zero element. (For more details on this aspect of $DR_K$, the reader is referred to our previous work \cite{Uramoto25}.)

Our construction of $DG_K$ starts from noting that the right-hand side of (\ref{pullback}) admits a natural non-abelian extension: we take the pullback of $\frob_\idp^\nat \leq G_\idp/I_\idp$, not along the quotient $G_\idp^{ab} \twoheadrightarrow G_\idp/I_\idp$ from $G_\idp^{ab}$, but along $q': G_\idp \twoheadrightarrow G_\idp/I_\idp$ from the absolute Galois group $G_\idp$. By this construction, the resulting monoid $G_\idp \times_{q'} \frob_\idp^\nat$ is a submonoid of the \emph{Weil group} $W_\idp$ of the local field $K_\idp$ \cite{Tate}, which we denote as $W_\idp^+ = G_\idp \times_{q'} \frob_\idp^\nat$ and call the \emph{Weil monoid} of $K_\idp$. Finally, as the local factor $O_\idp$ of $\widehat{O}_K$ could be reconstructed as the one-point compactification of $G_\idp^{ab} \times_q \frob_\idp^\nat \simeq O_\idp^*$, we construct a multiplicative monoid $Q_\idp := W_\idp^+ \cup \{0_\idp\}$ as the one-point compactification of $G_\idp \times_{q'} \frob_\idp^\nat = W_\idp^+$ where $0_\idp$ behaves as the zero element. We use these (profinite) monoids $Q_\idp$ as our non-abelianization of the local factors $O_\idp$ in $\widehat{O}_K \times_{\widehat{O}_K^\times} G_K^{ab} = DR_K$. 

To be precise, in order to remember the information of embeddings $\bar{K} \hookrightarrow \bar{K}_\idp$ in the construction of $DG_K$, we actually distinguish the conjugacy classes of the submonoid $G_\idp \hookrightarrow G_K$; thus, our new (non-abelianized) local factors $Q_\idp$ are indexed, not by the primes $\idp \in P_K$ of $K$, but by the primes of the algebraic closure $\bar{K}$. To this end, let $\V_K$ be the set of (finite) primes of $\bar{K}$, which is isomorphic to the set of maximal closed subgroups of $G_K$ of MLF-type in the sense of Hoshi, cf.\ Lemma 3.4 \cite{Hoshi}\footnote{In the notation of \cite{Hoshi}, our $\V_K$ corresponds to $\V_{\tilde{F}}$ in \cite{Hoshi}; we omit this tilde just for simplicity. In particular, do not confuse our $\V_K$ with Hoshi's $\V_F$ that represents the set of \emph{conjugacy classes} of elements of $\V_{\tilde{F}}$.}, hence we shall often identify them. To each $v \in \V_K$, associated is the closed subgroup $G_v \leq G_K$ of MLF-type that corresponds to the embedding $\bar{K} \hookrightarrow \bar{K}_v$; also, we shall denote by $\idp_v \in P_K$ for each $v \in \V_K$ the restriction of $v$ to $K$. The inertia subgroup of $G_v$ is denoted by $I_v$, whence the quotient $G_v/I_v$ is topologically generated by the Frobenius $\frob_v$; using the quotient map $q_v: G_v \twoheadrightarrow G_v/I_v$, we pullback the submonoid $\frob_v^\nat \leq G_v/I_v$ to define the \emph{Weil monoid} $W_v^+ := G_v \times_{q_v} \frob_v^\nat$. With these notations, we introduce our target non-abelian local factor $Q_v$: 

\begin{defn}[non-abelian local factor $Q_v$]
For each $v \in V_K$, we define a multiplicative profinite monoid $Q_v$ by the following:
\begin{eqnarray}
 Q_v &:=& W_v^+ \cup \{0_v\}
\end{eqnarray}
with the topology of one-point compactification of $W^+_v$, where $0_v$ behaves as the zero element. 
\end{defn}

\begin{lem}[maximal abelianization of $Q_v$]
\label{abelianization of Q_v}
For each $v \in \V_K$, the maximal abelian quotient $Q_v^{ab}$ of $Q_v$ is canonically isomorphic to $O_{\idp_v}$:
\begin{eqnarray}
Q_v^{ab} &\simeq& O_{\idp_v}.
\end{eqnarray}
\end{lem}
\begin{proof}
The proof is done just by comparing the definition of $W_v^+ = Q_v \setminus \{0_v\}$ with the presentation (\ref{pullback}) of $O^*_{\idp_v} = O_{\idp_v} \setminus \{0\}$; in particular, we have a canonical isomorphism $(W_v^+)^{ab} = O_{\idp_v}^*$. 
\end{proof}

\begin{rem}[unit group of $Q_v$]
For each $v \in \V_K$, the unit group $Q_v^\times$ of the profinite monoid $Q_v$ is the inertia subgroup $I_v$ of $G_v$. In particular, under the maximal abelian quotient $Q_v \twoheadrightarrow Q_v^{ab} = O_{\idp_v}$, the inertia subgroup $I_v = Q_v^\times$ is not only mapped onto the unit group $O_{\idp_v}^\times$, but also the inverse of the unit group $O_{\idp_v}^\times$ of $O_{\idp_v}$ by the quotient $Q_v \twoheadrightarrow O_{\idp_v}$ is exactly $I_v$. In this respect, we shall replace $O_{\idp_v}^\times$ with the inertia subgroup $I_v \leq G_v$. (See also (5) of Theorem 1.4 \cite{Hoshi}.)
\end{rem}

With these local factors $Q_v$, we now construct our non-abelian alternative of the multiplicative monoid $\widehat{O}_K = \prod_\idp O_\idp$. Naively one may replace $\widehat{O}_K$ with the product $\prod_v Q_v$, but the commutativity among the local factors $Q_v$ in $\prod_v Q_v$ induces an issue that contradicts to our goal (i.e.\ $DG_K^\times = G_K$) as mentioned below; thus, we shall replace $\widehat{O}_K = \prod_\idp O_\idp$, not with the product $\prod_v Q_v$, but with a certain \emph{free product} of $Q_v$'s that we review briefly. 

As in the case of free products of profinite groups (cf.\ \cite{Neukirch_cohomology}), we can define a free product of any family of profinite monoids and prove its existence/uniqueness up to isomorphism, although there are several possible variants of definitions. Since we need a certain specific form of a free product of the above profinite monoids $Q_v$ for our construction of $DG_K$, let us define this concept firstly in full generality as follows:

\begin{defn}[free product $\Asterisk_i D_i$]
Let $\{D_i\}_{i \in I}$ be a family of profinite monoids $D_i$. Then, a \emph{free product of $\{D_i\}_i$} is a profinite monoid $D$ equipped with (continuous) homomorphisms $\iota_i: D_i \rightarrow D$ such that for any family $\{\phi_i: D_i \rightarrow D'\}$ of homomorphisms, there exists a unique homomorphism $\phi: D \rightarrow D'$ such that $\phi_i = \phi \circ \iota_i$ for all $i \in I$. 
\end{defn}

\begin{rem}
The above definition of free product of profinite monoids does not coincide with that given in \cite{Neukirch_cohomology} in the case where $D_i$'s are all profinite groups; there, a certain finiteness condition (i.e.\ \emph{convergence}) on the family $\{\iota_i\}_i$ is required (cf.\ Definition 4.1.1 \cite{Neukirch_cohomology}). Formally, it makes sense also for families of profinite monoids; but we removed this finiteness condition since, with this condition, our construction of $DG_K$ results in a trivial monoid; technically, if $Q_v$ becomes trivial for all but fimitely many $v \in \V_K$, then $Q_v$ gets trivial for all $v \in \V_K$ in any finite quotient due to the fact that $Q_v$'s over some single $\idp \in P_K$ are all Galois conjugate, while such $v$'s constitute an infinite set. To avoid this issue, we will impose a slightly weaker condition on our construction of free product. 
\end{rem}

\begin{prop}[unique existence of free product]
For any family $\{D_i\}_{i \in I}$ of profinite monoids $D_i$, its free product exists uniquely up to canonical isomorphism; we denote it as $\Asterisk_i D_i$. 
\end{prop}
\begin{proof}
The proof is essentially the same as that of Chapter IV, \cite{Neukirch_cohomology}. Just to make the construction explicit, we sketch the proof. The uniqueness is easy to see if it exists; we prove only its existence. To construct the target $\Asterisk_i D_i$, let $F$ be the free product of monoids $D_i$ as a discrete monoid; then, let $D$ be the completion of $F$ with the family of finite quotients $F \twoheadrightarrow M$ such that $D_i \rightarrow F \twoheadrightarrow M$ are all continuous; we see that $D$ is our target one. By definition, there exist (continuous) homomorphisms $\iota_i: D_i \rightarrow F \rightarrow D$; to see their universality, let $\phi_i: D_i \rightarrow D'$ be (continuous) homomorphisms. For each finite quotient $\pi_H: D' \twoheadrightarrow H$, we have the compositions $\pi_H \circ \phi_i: D_i \rightarrow H$, from which we get a unique homomorphism $\phi_H: D \rightarrow H$ such that $\pi_H \circ \phi_i = \phi_H \circ \iota_i$ by definition of $D$. Then these $\phi_H: D \rightarrow H$ indexed by finite quotients $D' \twoheadrightarrow H$ induce a homomorphism $\phi: D \rightarrow D'$ due to the profiniteness of $D'$, which proves the universality of $D$ and $\iota_i: D_i \rightarrow D$ as requested. 
\end{proof}

In order to define our non-abelian alternative of $\widehat{O}_K$, denote by $\V_S \subseteq \V_K$ for each finite subset $S \subseteq P_K$ the set of those $v \in \V_K$ with $\idp_v \in S$. For each $S \subseteq P_K$, we have the free product $\Asterisk_{v \in \V_S} Q_v$, which we denote by $Q_S$; if $S \subseteq S'$, there exists a canonical surjective homomorphism $q^{S'}_S: Q_{S'} \twoheadrightarrow Q_S$ so that $q^{S'}_S (\rho_v) = 1$ for all $\rho_v \in Q_v$ with $\idp_v \in S' \setminus S$. These $Q_S$ together with $q^{S'}_S$ constitute an inverse system of profinite monoids, with which we define our non-abelian variant $\widehat{Q}_K$ of $\widehat{O}_K$: 

\begin{defn}[$\widehat{Q}_K$]
We define a profinite monoid $\widehat{Q}_K$ as the inverse limit of $Q_S$ where $S \subseteq P_K$ ranges over all finite subsets of $P_K$: That is, 
\begin{eqnarray}
 \widehat{Q}_K &:=& \lim_{\leftarrow S \subseteq P_K} Q_S. 
\end{eqnarray}
\end{defn}

\begin{rem}[Galois action on $\widehat{Q}_K$]
Unlike the case of $\widehat{O}_K$, there exists a natural action of $G_K$ on the profinite monoid $\widehat{Q}_K$ induced from that on $\V_K$: First, the conjugate $G_K \ni \tau \mapsto \sigma \tau \sigma^{-1} \in G_K$ for each $\sigma \in G_K$ induces a (left) action of $G_K$ on the set $\V_K$ (of maximal closed subgroups of MLF-type), which we denote as $\V_K \ni v \mapsto \sigma v \in \V_K$; also, since the Weil monoid $W_v^+$ is a submonoid of $G_v \leq G_K$, the same action then induces the transformation $W_v^+ \ni \rho \mapsto {}^\sigma \rho := \sigma \rho \sigma^{-1} \in W_{\sigma v}^+$, which extends to $Q_v \rightarrow Q_{\sigma v}$ by sending $0_v$ to $0_{\sigma v}$. This transformation on $Q_v$'s then induces an action of $G_K$ on the profinite monoid $Q_S = \Asterisk_{v \in \V_S} Q_v$ by the universality of free product, compatibly for all finite $S \subseteq P_K$, so that the action induces that on $\widehat{Q}_K = \lim_S Q_S$ too.  
\end{rem}

We now proceed to construct $DG_K$ using the above $\widehat{Q}_K$ and $G_K$. Let us denote by $\widehat{Q}_K \ltimes G_K$ the semi-direct product of $\widehat{Q}_K$ and $G_K$ with respect to the above Galois action. The free-product profinite group $\widehat{I}_K := \Asterisk_{v \in \V_K} I_v$ acts on $\widehat{Q}_K$ from the right via the natural map $I_v \rightarrow Q_v \rightarrow \widehat{Q}_K$, and on $G_K$ from the left via the inverse of the inclusion $I_v \hookrightarrow G_K$, i.e.\ $I_v \ni u_v \mapsto u_v^{-1} \in G_K$. Then our target $DG_K$ is defined by taking the profinite quotient of $\widehat{Q}_K \ltimes G_K$ with respect to this action of the group $\widehat{I}_K$:

\begin{defn}[free deformation of $G_K$]
By the \emph{free deformation of $G_K$}, we mean the quotient of the semi-direct product $\widehat{Q}_K \ltimes G_K$ by the minimal profinite congruence including the relation $(\rho, \sigma) \sim (\rho u_v, u_v^{-1} \sigma)$ for each $\rho \in \widehat{Q}_K, u_v \in I_v$ and $\sigma \in G_K$; we shall denote the resulting profinite monoid abusively by:
\begin{eqnarray}
\label{DG_K}
 DG_K &:=& \widehat{Q}_K \ltimes_{ \widehat{I}_K} G_K 
\end{eqnarray}
\end{defn}

\begin{rem}[why not $\prod_v Q_v$]
\label{why not direct product}
As mentioned above too, it is intentional that we used the free products in $\widehat{Q}_K = \lim_S \Asterisk_{v \in \V_S} Q_v$ rather than the direct product $\prod_v Q_v$: Although the construction $\bigl( \prod_v Q_v \bigr) \ltimes_{\widehat{I}_K} G_K$ itself makes similar sense, we can in fact show that the unit group of this profinite monoid descends to $G_K^{ab}$. Indeed, let us denote this profinite monoid as, say, $dG_K$; and suppose that $\phi: dG_K \twoheadrightarrow H$ be a finite quotient. Since the inertia subgroups $I_v \leq G_K$ generate (topologically) $G_K$, the images $\pi(1, u_v)$ of $[1, u_v] \in dG_K$ with $u_v \in I_v$ generate $\phi(G_K) \subseteq H$; moreover, as we later prove (using Chebotarev density theorem), we actually have $H^\times = \phi(G_K)$. On the other hand, by definition, we have $[1, u_v] = [u_v, 1]$ in $dG_K$, as well as $[u_v u_{v'}, 1] = [u_{v'} u_v, 1]$ for any $u_v \in I_v$ and $u_{v'} \in I_{v'}$ with $v \neq v'$ because of the commutativity of $Q_v$'s in $\prod_v Q_v$. From this, we can show that $H^\times = \phi(G_K)$ is abelian, and hence, so is $(dG_K)^\times$. (Actually, we can further prove that $dG_K^\times = G_K^{ab}$; but we do not need this fact in this paper.) Since this contradicts to our goal that the unit group $(DG_K)^\times$ of $DG_K$ should be $G_K$, the direct product $\prod_v Q_v$ cannot be our non-abelian alternative of $\widehat{O}_K$. 
\end{rem}

\begin{rem}[why not $\Asterisk_{v \in \V_K} Q_v$]
\label{why not over V_K}
We could define a profinite monoid using $\Asterisk_{v \in \V_K} Q_v$ instead of $\widehat{Q}_K = \lim_S \Asterisk_{v  \in \V_S} Q_v$. However the latter has an advantage that the \emph{Frobenius elements} $\psi_v \in DG_K$ ($v \in \V_K$) defined below generate topologically $DG_K$ as we see in the next subsection. For the time of writing this paper, we do not know whether the former variant also has this aspect. 
\end{rem}

\subsection{Basic properties of $DG_K$}
\label{s3.2}
As seen above, our construction of $DG_K$ was motivated by the fact that the local factors $O_\idp$ (minus zero) of $\widehat{O}_K$ in the adelic description $DR_K = \widehat{O}_K \times_{\widehat{O}_K^\times} G_K^{ab}$ could be naturally extended to the Weil monoid $W_v^+$; in this subsection we guarantee that this replacement preserves some basic properties of $DR_K$, yet extending the unit group $DR_K^\times = G_K^{ab}$ of $DR_K$ to the absolute Galois group $G_K$. Here, we prove two basic facts that (1) $DG_K$ is a deformation of $G_K$ in our sense (\S \ref{s1}); and (2) the maximal abelian quotient of $DG_K$ is isomorphic to $DR_K$. 

We start with introducing useful topological generators of $DG_K$, which are natural non-abelian analogue of the topological generators $\psi_\idp$ of $DR_K$. Indeed under the quotient $DG_K \twoheadrightarrow DR_K$, these generators will be identical to $\psi_\idp$'s; and also, similarly to $\psi_\idp \in DR_K$, their invertibility characterizes (un)ramification of primes in non-abelian extensions. 

\begin{lem}[topological generators $\psi_v$ of $DG_K$]
\label{generator of DG_K}
There exists a homomorphism $\psi: \V_K^* \rightarrow DG_K$ from the discrete monoid $\V_K^*$ of finite words over $\V_K$ with dense image. 
\end{lem}
\begin{proof}
For each $v \in \V_K$, choose $\gamma_v \in W_v^+ \leq G_v$ such that its image under the projection $G_v \twoheadrightarrow G_v/I_v$ is equal to $\frob_v \in G_v/I_v$; then, define $\psi_v := [\gamma_v, \gamma_v^{-1}] \in \widehat{Q}_K \ltimes_{\widehat{I}_K} G_K = DG_K$, which is independent of the choice of $\gamma_v$ due to the action of $\widehat{I}_K$. Then this map $v \mapsto \psi_v$ extends uniquely to a monoid homomorphism $\psi: \V_K^* \rightarrow DG_K$, which we show has a dense image. 

By definition, it is clear that $DG_K$ is generated (topologically) by $G_K$ and the elements $[\gamma_v, 1]$, where note also that each $\rho_v \in W_v^+$ is of the form $\rho_v = \gamma_v^n u_v$ with $u_v \in I_v$ and that $[u_v, 1] = [1, u_v]$; also, by Chebotarev density theorem, $G_K$ is generated by the lifts (or extensions) $\gamma_v \in G_K$ of $\frob_v$. Therefore, it follows that $DG_K$ is generated by the elements $[\gamma_v, 1]$ and $[1, \gamma_v^{-1}]$, hence equivalently, by $\psi_v = [\gamma_v, \gamma_v^{-1}]$ and $[1, \gamma_v]$. Thus, it suffices to prove that $G_K$ is generated by $\psi_v$'s. To see this, suppose that $\psi_v$ is invertible in a finite quotient $\pi: DG_K \rightarrow D$, whence we show that $\psi_v$ is equal to $[1, \gamma_v^{-1}]$ in $D$. (By Chebotarev density theorem, this will complete the proof.) By the invertibility of $\pi(\psi_v) \in D$, we know that $\pi(\psi_v)^\omega = \lim_n \pi(\psi_v)^{n!}$ converges to the identity $1 \in D$ (cf.\ \cite{Uramoto25}); on the other hand, we have:
\begin{eqnarray} 
 \psi_v^\omega &=& [0_v, 1]
\end{eqnarray}
because $Q_v = W_v^+ \cup \{0_v\}$ is the one-point compactification of $W_v^+$ with the zero element $0_v$, which is the unique non-trivial idempotent in $Q_v$; here we use also that $\psi_v^n = [\gamma_v^n, \gamma_v^{-n}]$. Combining these identities, we have:
\begin{eqnarray}
 \pi(0_v, 1) &=& \pi(\psi_v^\omega) \\
 &=& \pi (\psi_v)^\omega\\
 &=& \pi(1, 1). 
\end{eqnarray}
With this in mind, we deduce:
\begin{eqnarray}
 \pi(\psi_v) &=& \pi (1, 1) \cdot \pi(\gamma_v, \gamma_v^{-1}) \\
 &=& \pi (0_v, 1) \cdot  \pi(\gamma_v, \gamma_v^{-1}) \\
 &=& \pi (0_v, 1) \cdot \pi(1, \gamma_v^{-1}) \\
 &=& \pi (1, 1)  \cdot \pi(1, \gamma_v^{-1}) \\
 &=& \pi (1, \gamma_v^{-1}),
\end{eqnarray}
as requested. This proves that (1) the unit group $D^\times$ is generated by $G_K$; (2) every element of $G_K$ gets equal to some $\psi_v$ in $D^\times$ by Chebotarev density theorem, where we use the fact that $\psi_v$'s are automorphisms on $D$ for infinitely many $v \in \V_K$ (by the definition of $\widehat{Q}_K$); thus (3) $D$ is generated by the elements $\pi(\psi_v)$ for $v \in \V_K$. This completes the proof. 
\end{proof}

\begin{rem}[The injectivity of $\psi: \V_K^* \rightarrow DG_K$]
The above homomorphism $\psi: \V_K^* \rightarrow DG_K$ is actually injective, which can be proved by a similar (but a bit finer) construction as Theorem \ref{DG_K deforms G_K} below. But since this injectivity is not necessary below, we do not prove it in this paper.
\end{rem}

We now prove our first goal that $(DG_K)^\times$ is isomorphic to $G_K$ by constructing sufficiently many finite $DG_K$-sets; to this end, it is helpful to have a characterization of finite $DG_K$-sets in terms of finite sets equipped with actions of the above topological generators $\psi_v$ and $G_K$. In other words:

\begin{lem}[characterization of finite $DG_K$-sets]
\label{DG_K as a fundamental monoid}
Let $\C_K$ be the category of finite sets $S$ equipped with continuous right actions of $\V_K^*$ and $G_K$ such that the following axioms are satisfied:
\begin{enumerate}
\item we have $\sigma \cdot \psi_v  = \psi_{\sigma(v)} \cdot \sigma$ on $S$ for all $v \in \V_K$ and $\sigma \in G_K$; 
\item we have $\psi_v = \gamma_v^{-1}$ on $S\cdot \psi_v^\omega$ for all $v \in \V_K$, where the inertia subgroup $I_v \leq G_K$ acts trivially; 
\item there is a finite set $T \subseteq P_K$ such that the action of $\psi_v$ for $v \notin \V_T$ is an automorphism of $S$. 
\end{enumerate}
Then the pair $\langle \C_K, \F_K \rangle$, with $\F_K: \C_K \rightarrow \fsets$ the natural forgetful functor, forms a semi-galois category, whose fundamental monoid $\pi_1(\C_K, \F_K)$ is isomorphic to $DG_K$. 
\end{lem}
\begin{proof}
It suffices to prove a categorical equivalence $\cl DG_K \simeq \C_K$ that preserves the underlying sets of each $(S, \rho) \in \cl DG_K$. The functor $\Phi: \cl DG_K \rightarrow \C_K$ is defined by sending each $(S, \rho) \in \cl DG_K$ to the underlying set $S$ equipped with the action of $\psi_v \in DG_K$ and $\sigma \in G_K \subseteq DG_K$; we first see that $\Phi(S, \rho)$ satisfies the above axioms. The first axiom is clear from the definition of $\psi_v$ and $DG_K$; concerning the third, note that $DG_K$ can be given as the inverse limit of $Q_T \ltimes_{\widehat{I}_K} G_K$ where $T$ ranges over all finite subsets of $P_K$; thus, the action of $DG_K$ on $S$ factors through $Q_T \ltimes_{\widehat{I}_K} G_K$ for some finite $T \subseteq P_K$. Therefore, for all $v \notin \V_T$, the action of the factor $[\gamma_v, 1] \in DG_K$ acts trivially on $S$; thus $\psi_v = [\gamma_v, \gamma_v^{-1}]$ is equal to the automorphism by $[1, \gamma_v^{-1}] \in G_K$ on $S$. To see the second, recall that we have:
\begin{eqnarray}
 \psi_v^\omega &=& [0_v, 1].
\end{eqnarray}
Thus we have:
\begin{eqnarray}
 \psi_v^\omega \cdot \psi_v &=& [0_v, 1] \cdot [\gamma_v, \gamma_v^{-1}] \\
 &=& [0_v, \gamma_v^{-1}] \\
 &=& [0_v, 1] \cdot [1, \gamma_v^{-1}] \\
 &=& \psi_v^\omega \cdot \gamma_v^{-1}. 
\end{eqnarray}
This implies the first half of the second axiom; to see the second half, recall that we have $[u_v, 1] = [1, u_v]$ in $DG_K$ for $u_v \in I_v$; thus, we can deduce the following identities similarly to the above:
\begin{eqnarray}
 \psi_v^\omega \cdot [1, u_v] &=& \psi_v^\omega \cdot [u_v, 1] \\
 &=& \psi_v^\omega.
\end{eqnarray}
This shows that $I_v \leq G_K$ acts trivially on $S \cdot \psi_v^\omega$, as requested. 

It is easy to see that $\Phi: \cl DG_K \rightarrow \C_K$ is fully faithful. To see that $\Phi$ is essentially surjective, we take any $S \in \C_K$ and construct a right action $\rho$ of $DG_K$ on $S$ so that $\Phi(S, \rho) = S$. To this end, let $T \subseteq P_K$ be as in the third axiom; it suffices to construct a right action of $Q_T \ltimes_{\widehat{I}_K} G_K$ on $S$. In order to define an action of $Q_T \ltimes_{\widehat{I}_K} G_K$ on $S$, we need to define actions of $Q_v$ for $v \in \V_T$ and $G_K$ on $S$ so that (1) $I_v \leq W_v^+$ and $I_v \leq G_K$ act equally, and (2) the actions of $\sigma \cdot \rho_v$ and ${}^\sigma\rho_v \cdot \sigma$ coincide on $S$ for any $\rho_v \in Q_v$ and $\sigma \in G_K$. To construct an action of $Q_v$ on $S$, we choose a lift $\gamma_v \in W_v^+$ of Frobenius for each $v \in \V_T$; and take the unique representation $\rho_v = \gamma_v^n u_v$ for each $\rho_v \in W_v^+$, where $u_v \in I_v$; also, we regard $0_v = \gamma_v^\omega$. Then, we define:
\begin{eqnarray}
 s \cdot \rho_v &:=& s \cdot (\psi_v \cdot \gamma_v)^n \cdot u_v,
\end{eqnarray}
where the action of $\psi_v$ is the one equipped on the original $S \in \C_K$, where $G_K$ is already acting on $S$ by definition; note here that we have $(\psi_v \cdot \gamma_v)^n = \psi_v^n \cdot \gamma_v^n$, and that the action $s \cdot 0_v$ makes sense to coincide with $s \cdot \psi_v^\omega$. In particular, it follows that we could also define as $s \cdot \rho_v := s \cdot \psi_v^n \cdot \rho_v$\footnote{Here we abusively denote the action of $\rho_v \in Q_v$ and that of $\rho_v \in W_v \subseteq G_K$ by the same symbol.}, which means that this definition does not depend on the choice of $\gamma_v$. By this definition, it is clear that $u_v \in I_v \leq W_v^+$ and $u_v \in I_v \leq G_K$ act equally on $S$. To see that $\sigma \cdot \rho_v = {}^\sigma \rho_v \cdot \sigma$, we use the first axiom $\sigma \cdot \psi_v = \psi_{\sigma(v)} \cdot \sigma$, that is:
\begin{eqnarray}
  \sigma \cdot \rho_v &=& \sigma \cdot \psi_v^n \cdot \rho_v \\
  &=& \psi_{\sigma(v)}^n \cdot \sigma \cdot \rho_v \\
  &=& \psi_{\sigma(v)}^n \cdot {}^\sigma \rho_v \cdot \sigma\\
  &=& {}^\sigma \rho_v \cdot \sigma. 
\end{eqnarray}
The continuity of the actions of $Q_v = W_v^+ \cup \{0_v\}$ follows from the continuity of $G_K$-action on $S$ and the second axiom for objects in $\C_K$. By definition, it is clear that we have $\Phi(S, \rho) = S \in \C_K$. This completes the proof. 
\end{proof}

\begin{thm}[$DG_K^\times = G_K$]
\label{DG_K deforms G_K}
The unit group of $DG_K$ is isomorphic to $G_K$. 
\end{thm}
\begin{proof}
Since we have already proved that $DG_K^\times$ is a quotient of $G_K$ (see (1) in the last part of the proof of Lemma \ref{generator of DG_K}), it suffices to construct a finite $DG_K$-set $D_{L/K}$ for each finite Galois extension $L/K$ such that $G_K$ acts on $D_{L/K}$ precisely through $G_K \twoheadrightarrow Gal(L/K)$. For this aim, we shall apply our re-interpretation of $DR_\idf$'s discussed in Remark \ref{nonabelian DR_f}.

Let $L/K$ be a finite Galois extension; we denote by $\idf:=\idf_{L/K} \in I_K$ the non-zero integral ideal of $K$ defined in terms of higher ramification subgroups of $Gal(L/K)$ as follows: taking $\idP \in P_L$ over $\idp$, 
\begin{eqnarray}
 v_\idp(\idf) &:=& \min \{e \in \nat \mid Gal^e(L_\idP/K_\idp) = 1\};
\end{eqnarray}
since $L/K$ is now Galois, this $\idf$ does not depend on the choice of $\idP \mid \idp$. For each divisor $\idd \mid \idf$, let us denote by $L^{(\idf/\idd)}$ the field fixed by the subgroup of $Gal(L/K)$ generated by the higher ramification subgroups $Gal^{e_\idp}(L_\idP/K_\idp)$, where $e_\idp := v_\idp(\idf/\idd)$ and $\idP \in P_L$ over $\idp \in P_K$; we remark that $L^{(\idf/\idd)}$ is Galois over $K$ since we let $\idP$ run over all $\idP$ above $\idp$. Also, let us write as $G^{(\idf/\idd)} := Gal(L^{(\idf/\idd)}/K)$. With these notations, we shall define the underlying set of $D_{L/K}$ as follows:
\begin{eqnarray}
 D_{L/K} &:=& \coprod_{\idd \mid \idf} \{\idd \} \times G^{(\idf/\idd)};
\end{eqnarray}
In this proof, we need this $D_{L/K}$ just as a finite ($DG_K$-) set. (But for our latter use (\S \ref{s5}), we shall also equip it with the monoid structure defined by:
\begin{eqnarray}
 (\idd, \sigma) \cdot (\idd', \sigma') &:=& (\idd'', \sigma \cdot \sigma'),
\end{eqnarray}
where $\idd'' := (\idf, \idd \idd')$; and for $\sigma \in G^{(\idf/\idd)}$ and  $\sigma' \in G^{(\idf/\idd')}$, we take their product in $G^{(\idf/\idd'')}$, which is well-defined. It is clear that $D_{L/K}^\times = Gal(L/K)$; we will later prove that there exists a continuous surjective homomorphism $\pi: DG_K \twoheadrightarrow D_{L/K}$. In other words, the $DG_K$-set $D_{L/K}$ that we are now constructing is actually a galois object of $\cl DG_K$.)

In order to define a continuous action of $DG_K$ on this finite set $D_{L/K}$, it suffices to construct a continuous right action of $G_K$ and those of $\psi_v$ satisfying the axioms of Lemma \ref{DG_K as a fundamental monoid}. The right $G_K$ action on $D_{L/K}$ is the natural one, which factors precisely through $G_K \twoheadrightarrow Gal(L/K)$; to define the action of $\psi_v$ for $v \in \V_K$, we set as follows: for each $(\idd, \sigma) \in D_{L/K}$, 
\begin{eqnarray}
 (\idd, \sigma) \cdot \psi_v &:=& \left\{ \begin{array}{ll} (\idd \idp_v, \sigma) & \textrm{if $\idd \idp_v \mid \idf$} \\ (\idd, \sigma \cdot \frob_v^{-1}) & \textrm{if $\idd \idp_v \not \mid \idf$.} \end{array}\right.
\end{eqnarray}
Notice that, if $\idd \idp_v \not \mid \idf$, the prime $\idp_v$ is unramified in $L^{(\idf/\idd)}$, so that the inertia subgroup $I_v \leq G_v$ acts trivially on $L^{(\idf/\idd)}$, whence $\frob_v \in Gal(L^{(\idf/\idd)}/K)$ makes sense. We need to prove that this action of $\psi_v$'s satisfies the axioms of Lemma \ref{DG_K as a fundamental monoid}. But this is easy to prove: The first axiom $\tau \cdot \psi_v = \psi_{\tau(v)} \cdot \tau$ follows from the above definition (when $\idd \idp_v\mid \idf$) and the following general identity (when $\idd \idp_v \not \mid \idf$):
\begin{eqnarray} 
 \tau \cdot \frob_v^{-1} &=& \frob_{\tau(v)}^{-1} \cdot \tau.
\end{eqnarray}
The second axiom also follows from the above definition and the fact that $v_{\idp_v}(\idf)$ is finite; finally, the third axiom follows from the fact that $\idf$ contains only finitely many prime factors. This completes the proof. 
\end{proof}

\begin{thm}[$DR_K = (DG_K)^{ab}$]
\label{abelian quotient of DG_K}
We have a natural isomorphism of profinite monoids:
\begin{eqnarray}
 (DG_K)^{ab} &\simeq& DR_K,
\end{eqnarray}
where $(DG_K)^{ab}$ denotes the maximal abelian quotient of $DG_K$ in such a manner that the following diagram commutes:
\begin{equation}
\label{compatibility}
\xymatrix{
 \V_K^* \ar[r]^\psi \ar@{->>}[d] & DG_K \ar@{->>}[d] \\
 I_K \ar[r] & DR_K
}
\end{equation}
\end{thm}
\begin{proof}
This is almost trivial from the definition of $DG_K$ and Lemma \ref{abelianization of Q_v}; thus, we just sketch some relevant constructions. In particular, we define a surjective continuous homomorphism $\pi: DG_K \twoheadrightarrow DR_K$; then, see that every finite quotient $DG_K \twoheadrightarrow D$ onto a commutative finite monoid $D$ factors through this map $\pi: DG_K \twoheadrightarrow DR_K$. 

To this end, let $\phi_v: Q_v \rightarrow DR_K$ for $v \in V_K$ be defined by the composition $Q_v \twoheadrightarrow O_{\idp_v} \hookrightarrow DR_K$, where the first surjection $Q_v \twoheadrightarrow O_{\idp_v}$ is the maximal abelianization (cf.\ Lemma \ref{abelianization of Q_v}); by definition of $Q_S = \Asterisk_{v \in \V_S} Q_v$, the family $\{\phi_v: Q_v \rightarrow DR_K\}_{v \in \V_S}$ induces a homomorphism $\phi_S: Q_S \rightarrow DR_K$ in a way compatible for each finite $S \subseteq P_K$; thus, with these $\phi_S$, we obtain $\phi: \widehat{Q}_K \rightarrow DR_K$. (Here note that, for every finite quotient $DR_K \twoheadrightarrow D$, all but finitely many $\idp \in P_K \subseteq DR_K$ are invertible in $D$, thus, the uniformizer $\pi_\idp \in O_\idp$ gets trivial in $D$ under this quotient $DR_K \twoheadrightarrow D$.) By the projection $G_K \twoheadrightarrow G_K^{ab}$, moreover, we have $G_K \twoheadrightarrow G_K^{ab} = DR_K^\times \hookrightarrow DR_K$. Combining these maps $\widehat{Q}_K \rightarrow DR_K$ and $G_K \rightarrow DR_K$, we obtain a well-defined homomorphism $DG_K = \widehat{Q}_K\ltimes_{\widehat{I}_K} G_K \rightarrow DR_K$, putting $[\rho, \sigma] \mapsto [\phi(\rho), \bar{\sigma}]$, where $\bar{\sigma} \in G_K^{ab}$ denotes the canonical projection of $\sigma \in G_K$.

The surjectivity is easy to see from the construction; to see that $DG_K \twoheadrightarrow DR_K$ is the maximal abelian quotient, let $DG_K \twoheadrightarrow D$ be a finite abelian quotient. First, by the commutativity of $D$, the conjugate action of $G_K$ on $\widehat{Q}_K$ gets trivial in $D$, which implies that, for galois conjugate $v, v' \in \V_K$, the images of $Q_v$ and $Q_{v'}$ in $D$ are identified. Since $\V_K/G_K \simeq P_K$ (cf.\ \cite{Hoshi}) and the maps $Q_v \rightarrow D$ and $G_K \rightarrow D$ that restrict $DG_K \rightarrow D$ onto the local factor $Q_v$ and $G_K = DG_K^\times$ respectively factor through $O_{\idp_v} \rightarrow D$ (Lemma \ref{abelianization of Q_v}) and $G_K^{ab} \rightarrow D$, it readily follows that $DG_K \twoheadrightarrow D$ factors through $DR_K = \widehat{O}_K \times_{\widehat{O}_K^\times} G_K^{ab} \twoheadrightarrow D$. Finally, the commutativity of the above diagram (\ref{compatibility}) follows from the facts that $\gamma_v \in Q_v$ is mapped to (some) uniformizer $\pi_{\idp_v} \in O_{\idp_v}$ and that $\gamma_v^{-1} \in W_v^+ \leq G_v \leq G_K$ is mapped to $[\pi_{\idp_v}]^{-1} \in G_K^{ab}$; thus, $\psi_v \in DG_K$ is mapped by $DG_K \twoheadrightarrow DR_K$ to $\idp_v = [\pi_{\idp_v}, \pi_{\idp_v}^{-1}] \in DR_K$ by construction. This completes the proof. 
\end{proof}

\section{Arithmetic of $DG_K$}
\label{s4}
Recall that the category $\cl DR_K$ of finite $DR_K$-sets is (opposite) equivalent to that of finite etale $\Lambda$-rings over $K$ having integral models (cf.\ \cite{Borger_Smit11}); this then implies that the algebraic Witt vectors can be identified with locally constant $K^{ab}$-valued $G_K$-equivariant functions on $DR_K$ (cf.\ \cite{Uramoto21}), yielding an arithmetic interpretation of $DR_K$. The subject of this section is to give a similar interpretation of $DG_K$, relating it with $\Lambda_{O_v}$-rings finite etale over $K_v$ having integral models at all $v \in \V_K$ in the sense of \cite{Borger_Smit11} (\S \ref{s4.1}). This motivates us to study in more detail the semi-galois category $\C_{K_\idp}$ of $\Lambda_{O_\idp}$-rings finite etale over $K_\idp$ having integral models \cite{Borger_Smit11}; and in \S \ref{s4.2}, we prove that (1) the fundamental monoid of $\C_{K_\idp}$ is isomorphic to $Q_\idp \ltimes_{I_\idp} G_\idp$; and (2) the galois objects of $\C_{K_\idp}$ are generated by some $\idp$-typical Witt vectors constructed from the fields of norms of the local field $K_\idp$ (cf.\ \S \ref{s4.2}). 

\subsection{Geometry of finite $DG_K$-sets}
\label{s4.1}
Let $\Cl DG_K$ be the category of (not necessarily finite) right $DG_K$-sets and $DG_K$-equivariant maps; then $\Cl DG_K$ forms a topos, whose \emph{coherent objects} constitute the category $\cl DG_K$ of finite $DG_K$-sets (cf.\ \S 6 \cite{Uramoto17}). For each $S \in \Cl DG_K$, we can restrict the action of $DG_K$ to that of $G_K = DG_K^\times$\footnote{We shall regard the right action of $\sigma \in G_K$ also as the left action of $\sigma^{-1} \in G_K$; in particular, the Galois action on finite (Galois) extensions $L/K$ is considered as the left action. This convention is related to (and compatible with) the fact that $G_K$ acts on $DR_K \simeq \widehat{O}_K \times_{\widehat{O}_K^\times} G_K^{ab}$ via $\sigma \mapsto [1, \sigma^{-1}]$ (cf.\ Proposition 8.2 \cite{Yalkinoglu}).}, which induces a (forgetful) functor $U=f^*: \Cl DG_K \rightarrow \Cl G_K$. We first show that this functor $U$ is \emph{comonadic}; in fact, this holds true in the following generality: 

\begin{prop}
\label{comonadicity}
Let $f: G \rightarrow D$ be a homomorphism between profinite monoids $G$ and $D$\footnote{Here $G$ is not necessarily a profinite group; with this symbol $G$, we just intend to highlight how this proposition applies to our current inclusion $G_K \hookrightarrow DG_K$. In particular, it follows that $\cl DG_K$ is comonadic over $\cl G_K$ (in the abusive sense of Definition \ref{comonad on cl G_K}), thus, $\cl DG_K$ can be dually identified with the category of modules over finite etale $K$-algebras for some monad on them. As we will see below, these modules has a description in terms of $\idp_v$-typical Witt vectors locally for all $v \in \V_K$, cf.\ Remark \ref{comparison with lambda rings}.}. Then the induced functor $U=f^*: \Cl D \rightarrow \Cl G$ has a right adjoint $V: \Cl G \rightarrow \Cl D$ such that $\Cl D$ is canonically equivalent to the category of comodules of the associated comonad $U V: \Cl G \rightarrow \Cl G$.\footnote{In terms of topos theory, this means that $f: G \rightarrow D$ induces a \emph{surjective} geometric morphism $f: \Cl G \twoheadrightarrow \Cl D$; but in this paper, we do not need this topos-theoretic terminology. (That said, see \S 3.2 \cite{Uramoto22} for some general intuition.)} 
\end{prop}
\begin{proof}
We just sketch some relevant constructions; the rest proofs are straightforward. Firstly, the right adjoint $V: \Cl G \rightarrow \Cl D$ is defined by assigning to each $G$-set $S \in \Cl G$ the following $D$-set:
\begin{eqnarray}
 V(S) &:=& \Hom_G (D, S),
\end{eqnarray}
where the right-hand side denotes the set of continuous $G$-equivariant maps from $D$ to the (discrete) set $S$, where $D$ has the right $G$-action via $f: G \rightarrow D$; this $V(S)$ is equipped with the right action of $D$ given by: for each $\alpha \in \Hom_G(D, S)$ and $d \in D$, 
\begin{eqnarray}
 \alpha \cdot d &:=& [d' \mapsto \alpha (d d') ],
\end{eqnarray}
Then $\alpha \cdot d \in \Hom_G(D, S)$; and the functor $V: \Cl G \rightarrow \Cl D$ is indeed a right adjoint to $U: \Cl D \rightarrow \Cl G$ with the following counit $\epsilon_S: S \rightarrow VU(S)$ for each $S \in \Cl D$: 
\begin{eqnarray}
 S &\xrightarrow{\epsilon_S}& VU(S) = \Hom_G(D, S)\\
 s &\longmapsto& \epsilon_S(s) := [d \mapsto s \cdot d].
\end{eqnarray}
To highlight some intuition, we shall denote the associated comonad abusively as $\D: \Cl G \rightarrow \Cl G$ (see also Example 3.2.3, \S 3.2 \cite{Uramoto22}). 

Let $(\Cl G)^\D$ be the category of $\D$-comodules; we give an equivalence from $\Cl D$ to $(\Cl G)^\D$. For each $D$-set $S \in \Cl D$, we define the corresponding $\D$-comodule $\nabla: S \rightarrow \D(S)$ as follows:
\begin{eqnarray}
 S &\xrightarrow{\nabla}& \D(S) = \Hom_G(D, S)\\
 s &\longmapsto& \nabla s := [d \mapsto \nabla_d (s):= s \cdot d]. 
\end{eqnarray}
We denote this $\D$-comodule as $\Phi(S) \in (\Cl G)^\D$. Conversely, for each $\D$-comodule $\nabla: S \rightarrow \D(S) \in (\Cl G)^\D$, we define the corresponding $D$-set $\Psi(S) \in \Cl D$ by: for each $s \in S$ and $d \in D$, 
\begin{eqnarray}
 s \cdot d &:=& (\nabla s) (d); 
\end{eqnarray}
we denote the right-hand side also as $\nabla_d (s)$. It is easy to see that the functors $\Phi: \Cl D \rightarrow (\Cl G)^\D$ and $\Psi: (\Cl G)^\D \rightarrow \Cl D$ give inverses to each other, hence the equivalence $\Cl D \simeq (\Cl G)^\D$. 
\end{proof}

\begin{defn}[comonad on $\cl G_K$]
\label{comonad on cl G_K}
We shall say that $\D: \Cl G \rightarrow \Cl G$ gives a comonad on $\cl G$ abusively for short, although $\D$ is not closed within $\cl G$ in general. Also, we shall mean by (finite etale) \emph{$\D$-modules} over $K$ the duals of $\D$-comodules with finite underlying sets for the comonad $\D: \Cl G_K \rightarrow \Cl G_K$ induced by $G_K \hookrightarrow DG_K$. 
\end{defn}

\begin{rem}[finite $DG_K$-sets are dual to finite etale $\D$-modules over $K$]
With this terminology, we shall often identify a finite $DG_K$-set $S \in \cl DG_K$ with a finite etale $\D$-module over $K$ whose underlying $K$-algebra is $X_S: = \Hom_{G_K}(S, \bar{K})$; in particular the category $\cl DG_K$ is often identified with that of finite etale $\D$-modules over $K$. 
\end{rem}

\begin{rem}[comparison with the $\Lambda$-rings of \cite{Borger_Smit08, Borger_Smit11}]
\label{comparison with lambda rings}
By Proposition \ref{comonadicity} applied to the inclusion $G_K^{ab} \hookrightarrow DR_K$, we obtain a comonad on $\cl G_K^{ab}$, thus dually, the concept of finite etale $\D$-modules with abelian components. Then the $\Lambda$-rings of \cite{Borger_Smit11} can be identified with these $\D$-modules; and can be seen as a special case of $\D$-modules for $G_K \hookrightarrow DG_K$. In this sense, the above $\D$-modules can be regarded as non-abelian generalizations of $\Lambda$-rings. In fact, these $\D$-modules can be described by familiar objects in that, localizing at each $v \in V_K$, they are precisely $\Lambda_{O_v}$-rings finite etale over $K_v$ having integral models in the sense of \cite{Borger_Smit11}; in other words, our $\D$-modules are characterized just as finite etale $K$-algebras which locally have the structure of $\Lambda_{O_v}$-rings in $\C_{K_v}$ of \cite{Borger_Smit11} as we are going to discuss below. 
\end{rem}

To see this, let $S \in \cl DG_K$ be a finite $DG_K$-set; and denote by $X_S \in \Et_K$ the dual finite etale $K$-algebra, namely, $X_S := \Hom_{G_K}(S, \bar{K})$. By restricting the action of $G_K$ to the subgroup $G_v$ for each $v \in \V_K$, we have a finite $G_v$-set, denoting $S_v$ (identical to $S$ as a finite set), which is dual to the finite etale $K_v$-algebra $X_{S,v} := \Hom_{G_v}(S_v, \bar{K}_v)$ isomorphic to the localization $X_S \otimes_K K_v$ of $X_S$ over the local field $K_v$. 

This localization clarifies the above-mentioned relationship between finite $DG_K$-sets and $\Lambda_{O_\idp}$-rings finite etale over $K_\idp$ with integral models: Globally, on the one hand, by the non-commutativity of $DG_K$, the action of $\psi_v$ on $S$ does not commute with that of $G_K$; thus, the action of $\psi_v$ does not give a $K$-algebra endomorphism on $X_S$ in general. But locally, on the other hand, the action of $\psi_v$ on $S_v$ commutes with the action of $G_v$, hence, $\psi_v$ gives rise to a $K_v$-algebra endomorphism on $X_{S, v}$. We shall denote it abusively as $\psi_v: X_{S, v} \rightarrow X_{S, v}$; then we can show that this $\psi_v: X_{S, v} \rightarrow X_{S, v}$ is an object of the semi-galois category $\C_{K_v}$ in the sense of \cite{Borger_Smit11}. In fact, this property even \emph{characterizes} finite $DG_K$-sets (and thus $DG_K$). 

To be precise, recall Lemma \ref{DG_K as a fundamental monoid} that finite $DG_K$-sets are characterized as finite sets equipped with actions of $\V_K^*$ (hence, of its profinite completion $\widehat{\V}_K^*$) and $G_K$ satisfying three axioms therein. In particular, the first axiom just says that the action is that of the semi-direct product $\widehat{\V}_K^* \ltimes G_K$ naturally defined; therefore, Lemma \ref{DG_K as a fundamental monoid} could be rephrased as the claim that the finite $DG_K$-sets are those finite $(\widehat{\V}_K^* \ltimes G_K)$-sets which satisfy the second and third axioms therein. Here note that, since $G_K$ embeds in $\widehat{\V}_K^* \ltimes G_K$, a finite $(\widehat{\V}_K^* \ltimes G_K)$-set $S$ induces a finite etale $K$-algebra $X_S$; and their localizations $X_{S, v}$ can be defined similarly as above; further, since $\psi_v$ and $\sigma \in G_v$ commute in this $\widehat{\V}_K^* \ltimes G_K$ too, the action of $\psi_v$ induces a $K_v$-algebra endomorphism $\psi_v: X_{S, v} \rightarrow X_{S, v}$ on the localized $K_v$-algebra $X_{S, v}$ as well. (That is, this construction makes sense at the level of $\widehat{\V}_K^* \ltimes G_K$ more than its quotient $DG_K$.) Therefore, the pair $(X_{S, v}, \psi_v)$ forms a $\Lambda_{O_v}$-ring over $K_v$; and thus, it makes sense to ask whether $(X_{S, v}, \psi_v)$ has an \emph{integral model} in the sense of \cite{Borger_Smit11}. That is:

\begin{defn}[integral model]
Let $S$ be a finite $\widehat{\V}_K^* \ltimes G_K$-set; and $X_{S, v}, \psi_v$ as above. An \emph{integral model of $X_S$ at $v \in \V_K$} is a finite $O_{\idp_v}$-subalgebra $A_v \subseteq X_{S, v} = X_S \otimes K_v$ such that:
\begin{enumerate}
\item $X_{S, v} = K_v \otimes A_v$ (taken over $O_{\idp_v}$); 
\item $\psi_v: X_{S, v} \rightarrow X_{S, v}$ restricts to $A_v$, that is, $\psi_v (A_v) \subseteq A_v$; 
\item the pair $(A_v, \psi_v)$ forms a $\Lambda_{\idp_v}$-ring over $O_{\idp_v}$ in the sense of \cite{Borger_Smit11}. 
\end{enumerate}
Further, we say that $X_S$ is \emph{unramified at $\idp \in P_K$} if $\psi_v: X_{S, v} \rightarrow X_{S, v}$ is an automorphism for some (and all) $v \in \V_K$ over $\idp$. (Here recall $\sigma \cdot \psi_v = \psi_{\sigma(v)} \cdot \sigma$.)
\end{defn}

\begin{thm}[arithmetic characterization of finite $DG_K$-sets]
A finite $\widehat{\V}_K^* \ltimes G_K$-set $S$ is a $DG_K$-set if and only if $X_S$ has integral models at every $v \in \V_K$, which are unramified at all but finitely many primes $\idp$. 
\end{thm}
\begin{proof}
This is a direct consequence of Lemma \ref{DG_K as a fundamental monoid} and Theorem 1.1 \cite{Borger_Smit11}. In fact, the first condition that $X_S$ has integral models at every $v \in \V_K$ is equivalent to the second axiom of Lemma \ref{DG_K as a fundamental monoid} by Theorem 1.1 \cite{Borger_Smit11}; also, the second condition that these integral models are unramified at almost all primes $\idp \in P_K$ is equivalent to the third axiom of Lemma \ref{DG_K as a fundamental monoid}. This completes the proof. 
\end{proof}

\subsection{Local description by fields of norms}
\label{s4.2}
In the sense proved above, the finite $DG_K$-sets are essentially those dual to finite etale $K$-algebras which are locally the objects of the semi-galois category $\C_{K_\idp}$ of finite etale $\Lambda_{O_\idp}$-rings over $K_\idp$ having integral models in the sense of Borger and de Smit \cite{Borger_Smit11}. This motivates us to study this semi-galois category $\C_{K_\idp} =: \C_\idp$ in more detail. 

The first main result in this subsection is a description of the fundamental monoid of $\langle \C_\idp, \F_\idp \rangle$; although this was already proved essentially in Lemma \ref{DG_K as a fundamental monoid}, it would be meaningful to state it as an independent theorem: 
\begin{thm}
Let $\C_\idp$ be the (semi-galois) category of finite etale $\Lambda_{O_\idp}$-rings over $K_\idp$ having integral models in the sense of \cite{Borger_Smit11}; and $\F_\idp: \C_\idp \rightarrow \fsets$ be its fiber functor. Then the fundamental monoid $\pi_1(\C_\idp, \F_\idp)$ is isomorphic to $Q_\idp \ltimes_{I_\idp} G_\idp$. 
\end{thm}
\begin{proof}
By Theorem 1.1 \cite{Borger_Smit11}, it suffices to see that the category of finite $(Q_\idp \ltimes_{I_\idp} G_\idp)$-sets is equivalent to that of those finite sets equipped with actions of $\psi_\idp^\nat \times G_\idp$ such that (i) the inertia subgroup $I_\idp$ acts trivially on $S^{unr} := S\cdot \psi_\idp^\omega$; and (ii) the action of $\psi_\idp$ on $S^{unr}$ is equal to that of $\frob_\idp \in G_\idp/I_\idp$. But this can be proved by the same argument as Lemma \ref{DG_K as a fundamental monoid}; in fact, we just need to restrict it to $\psi_\idp$ and $G_\idp$. 
\end{proof}

Our second goal is to construct (a cofinal family of) galois objects of $\C_\idp$ canonically in terms of fields of norms for the local field $K_\idp$, with which we provide yet another description of $Q_\idp \ltimes_{I_\idp} G_\idp$. In the following, let us denote as $DG_\idp := Q_\idp \ltimes_{I_\idp} G_\idp$ for short. For some details on fields of norms, the reader is referred to e.g.\ \cite{Wintenberger, Laubie, Benois}.

Fix an algebraic closure $\bar{K}_\idp$ of the local field $K_\idp$; let $\bar{K}_\idp \supseteq M \supseteq K_\idp$ be an infinite APF extension; we denote by $\FN_M := \FN(M/K_\idp)$ the field of norms of $M/K_\idp$. As is well-known, $\FN_M$ is of positive characteristic $p$ (equal to that of the residue field $k_\idp$ of $\idp$), thus, canonically possesses the \emph{Frobenius map} $\FN_M \ni \alpha \mapsto \alpha^p \in \FN_M$. Our basic idea to construct galois objects of $\C_\idp$ is to use this canonical Frobenius map (as well as the action of $G_\idp$) on $\FN_M$. For this purpose, we shall consider only infinite APF extensions which are Galois over $K_\idp$ and whose residue fields are finite over that of $K_\idp$. 

\begin{lem}[canonical construction of finite $DG_\idp$-sets]
\label{canonical DG_p sets}
Let $M/K_\idp$ be an infinite APF extension, whose residue field $k_M$ is of degree $d$ over $k_\idp = \fld_q$; $\FNO=\FNO_M$ the integer ring of $\FN_M$; and $\FNm=\FNm_M$ be the maximal ideal of $\FNO$. Then, for each $n$, the quotient ring $\FNO/ \FNm^{q^n}$ is finite and admits a natural action of $\psi_\idp$ and $G_\idp$ so that it forms a finite $DG_\idp$-set. 
\end{lem}
\begin{proof}
We put $S_n := \FNO/\FNm^{q^n}$; in addition to the natural action of $G_\idp$ on $S_n$, we can define the action of $\psi_\idp: S_n \rightarrow S_n$ by $\alpha \mapsto \alpha^q$. To prove the lemma, we need to show that (1) the inertia subgroup $I_\idp$ acts trivially on $S_n^{unr}:= S_n \cdot \psi_\idp^\omega$; and (2) the action of $\psi_\idp$ on $S_n^{unr}$ is equal to that of $\frob_\idp \in G_\idp/I_\idp$. 

For the first one, note that, using an isomorphic representation $\FNO \simeq \fld_{q^d}[[X]]$, the action of $\psi_\idp^\omega$ on $S_n$ maps any $\alpha = \alpha_0 + \alpha_1 X + \cdots$ to $\alpha_0$ modulo $\FNm^{q^n}$; hence, the image $S_n \cdot \psi_\idp^\omega \subseteq S_n$ consists of the constants $\fld_{q^d}$ only. Since the inertia subgroup $I_\idp$ acts trivially on constants, this proves (1). The second one is similar: Indeed, the action of $\psi_\idp$ is defined by the power $\alpha \mapsto \alpha^q$, which clearly coincides with the action of $\frob_\idp \in G_\idp/I_\idp$ on the constants $\fld_{q^d} = S_n \cdot \psi_\idp^\omega \subseteq S_n$. This completes the proof. 
\end{proof}

\begin{defn}[canonical finite quotients of $DG_\idp$]
For each infinite APF extension $M/K_\idp$ and $n \in \nat$, the image of $DG_\idp$ of the above action $DG_\idp \rightarrow \End(S_n)$ is denoted by $D_{M, n}$ with projection $DG_\idp \twoheadrightarrow D_{M, n}$.
\end{defn}

\begin{prop}[local structure of $D_{M, n}$]
\label{local structure of D_M, n}
The $\HH$-classes of $D_{M, n}$ are the $G_\idp$-orbits; and $D_{M, n}$ has the following $\HH$-class decomposition:
\begin{eqnarray}
\label{decomposition of D_M, n}
 D_{M, n} &=& \coprod_{0 \leq l \leq n} Gal(M_{n-l}/K_\idp),
\end{eqnarray}
where $M_i$ denotes the subfield of $M/K_\idp$ fixed by the $(q^i -1)$-th ramification subgroup in the \emph{lower} numbering. In particular, we have:
\begin{eqnarray}
 D_{M, n}^\times &=& Gal(M_n/K_\idp). 
\end{eqnarray}
\end{prop}
\begin{proof}
In this proof, put $D := D_{M, n}$. To see the structure of $\HH$-classes of $D$, we first see that the isotropy subgroup of $\psi_\idp^l \in D$ ($0 \leq l \leq n$) under the action of $G_\idp$ is the $(q^{n-l} - 1)$-th ramification subgroup (in the lower numbering), which we shall denote by $G_l$ for short; that is, $\sigma \cdot \psi_\idp^l = \psi_\idp^l$ in $D$ if and only if $\sigma \in G_l$. The equality $\sigma \cdot \psi_\idp^l = \psi_\idp^l$ means that, by definition of $D$ and $S_n \simeq \fld_{q^d}[[X]]/(X^{q^n})$, we have:
\begin{eqnarray}
  \sigma (\alpha^{q^l}) &\equiv& \alpha^{q^l} \mod X^{q^n}; 
\end{eqnarray}
for any $\alpha = \sum \alpha_i X^i \in \fld_{q^d}[[X]]$. Restricting this on the residue field $k_M$ by taking modulo $X$, this particularly implies that $\sigma \in I_\idp$; moreover, by applying to $\alpha = X$ in particular, this means that:
\begin{eqnarray}
 \sigma (X^{q^l}) &\equiv& X^{q^l} \mod X^{q^n}. 
\end{eqnarray}
To be more explicit, let $\sigma(X) = \sum_i \alpha_i X^i $ with $\alpha_1 = 1$. Then the above congruence is equivalent to:
\begin{eqnarray}
 \sum_{i\geq 1} (\alpha_i)^{q^l} X^{i q^l} &\equiv& X^{q^l} \mod X^{q^n}. 
\end{eqnarray}
Subtracting both sides by $X^{q^l}$, we obtain:
\begin{eqnarray}
  \sum_{i \geq 2} (\alpha_i)^{q^l} X^{i q^l} &\equiv& 0 \mod X^{q^n}. 
\end{eqnarray}
This is equivalent to that, if $i_0$ is the first index with $\alpha_i \neq 0$, we have $i_0 q^l \geq q^n$, that is, $i_0 \geq q^{n-l}$; in other words, $\sigma(X) \equiv X$ modulo $X^{q^{n-l}}$, which means that $\sigma \in G_l$ (cf.\ \S 3.3 \cite{Wintenberger}). From this, we can deduce that the decomposition in (\ref{decomposition of D_M, n}) gives the $G_\idp$-orbit decomposition. To prove that this is also the $\HH$-class decomposition, it suffices to note that $\psi_\idp^s$ and $\psi_\idp^t$ for distinct $s, t \geq 0$ cannot be $\HH$-equivalent unless $s, t \geq n$ because of $S_n \cdot \psi_\idp^s \neq S_n \cdot \psi_\idp^t$. This completes the proof. 
\end{proof}

\begin{thm}[field-of-norms description of $DG_\idp$]
\label{DG_p by fields of norms}
The finite monoids $D_{M, n}$ constitute an inverse system, whose limit is isomorphic to $DG_\idp$. That is:
\begin{eqnarray}
 DG_\idp &\simeq& \lim_{\leftarrow M, n} D_{M, n}. 
\end{eqnarray}
\end{thm}
\begin{proof}
By definition of $D_{M, n}$, we can see that they constitute a natural inverse system with respect to the  field extensions $M \subseteq M'$ and the (additive) ordering $n \leq n'$; we prove that they converges to $DG_\idp$. Since each $D_{M, n}$ is a finite quotient of $DG_\idp$ (Lemma \ref{canonical DG_p sets}), it suffices to see that the finite quotients $D_{M, n}$ separate distinct elements of $DG_\idp$ (or equivalently, any finite quotient $DG_\idp \twoheadrightarrow D$ factors through some $D_{M, n}$). 

Let $\phi: DG_\idp \twoheadrightarrow D$ be any finite quotient, whence $D$ admits natural actions of $\psi_\idp$ and $G_\idp$ which commute with each other. As the argument of Theorem 1.1 \cite{Borger_Smit11}, decompose $D$ as $D = \coprod_i^n D_i$, where $D_0 := D \cdot \psi_\idp^\omega$, and $D_i = \{s \in D \mid \textrm{$s \cdot \psi_\idp^i \in D_0$ and $s \cdot \psi_\idp^{i-1} \not \in D_0$} \}$, which are closed under the action of $G_\idp$. Moreover, since $D$ is now a galois object of $\C_\idp$, each $D_i$ is a $G_\idp$-orbit (that is, connected object of $\cl G_\idp$). Thus, for each $i$, we have a finite extension $L_i/K_\idp$ such that $D_i = \Hom_{K_\idp}(L_i, \bar{K}_\idp)$; and the dual of $D$ is of the form $X_D = L_0 \times L_1 \times \cdots \times L_n$. With this in mind, we now find some infinite (Galois) APF extension $M/K_\idp$ and $m \geq 0$ such that $D$ is a quotient of $D_{M, m}$. Since $\bar{K}_\idp$ is a union of infinite APF extensions whose residue fields are finite over that of $K_\idp$ (cf.\ Lemma 5, \S 3.1 \cite{Laubie}), we can find $M$ (Galois over $K_\idp$) so that $L_0 \subseteq M_0$, where $M_i$ denotes the one defined in Proposition \ref{local structure of D_M, n}; in particular, note that $M_0/K_\idp$ (as well as $L_0/K_\idp$) is unramified. Furthermore, since $M_i$ gets larger as $i$ gets larger, we can re-choose $M/K_\idp$ so that, for some $k_0$, we have $L_i \subseteq M_{k_0 + i}$ for any $0 \leq i \leq n$. Then, with this $M$ and $m = k_0 + n$, we see that $D$ is a quotient of $D_{M, m}$. 

To prove this, we first note the decompositions $D = \coprod_i D_i$ and $D_{M, m} = \coprod_j S_j$ as finite $G_\idp$-sets, where $D_i = \Hom_{K_\idp} (L_i, \bar{K}_\idp)$ and $S_j := \Hom_{K_\idp}(M_j, \bar{K}_\idp)$; and by construction, we have the natural surjections $\phi_{k_0+i}: S_{k_0+i} \twoheadrightarrow D_i$ for $0 \leq i \leq n$ of finite $G_\idp$-sets, which give a map $\coprod_{j \geq k_0} S_j \twoheadrightarrow D$; for $0 \leq j \leq k_0$ too, we define $\phi_j: S_j \rightarrow D_0 \subseteq D$ by $\phi_j (\sigma_j) := \frob_\idp^{k_0 - j} \cdot \bar{\sigma}_j$, where $\bar{\sigma}_j \in \Hom_{K_\idp}(L_0, \bar{K}_\idp)$ denotes the restriction of $\sigma_j \in \Hom_{K_\idp}(M_j, \bar{K}_\idp)$ onto $L_0 \subseteq M_j$; we remark that $\frob_\idp$ makes sense on $L_0$ because $L_0/K_\idp$ is unramified. Totally, we have constructed $\phi_j: S_j \rightarrow D$ for any $0 \leq j \leq m$, hence $\phi: D_{M, m} = \coprod_j S_j \rightarrow D$; and by definition, it is straightforward to see that this map $\phi$ is in fact surjective and \emph{$DG_\idp$-equivariant} (then $\phi: D_{M, m} \twoheadrightarrow D$ is also a monoid homomorphism as they are both galois objects of $\cl DG_\idp$). This completes the proof. 
\end{proof}

\begin{rem}[construction of algebraic $\idp$-typical Witt vectors]
By the same argument as \S 3.1 \cite{Uramoto21}, we can prove that the algebraic $\idp$-typical Witt vectors (that generate $\Lambda_{O_\idp}$-rings in $\C_\idp$) are precisely the locally constant $\bar{K}_\idp$-valued $G_\idp$-equivariant functions on $DG_\idp$; and Theorem \ref{DG_p by fields of norms} above implies that they are precisely the $\bar{K}_\idp$-valued $G_\idp$-equivariant functions on finite quotients $D_{M, n}$ of $DG_\idp$. By their origin, the latter functions can be canonically (or tautologically) constructed by each element $\alpha \in \FN_M$.

To this end, fix an infinite APF extension $M/K_\idp$, and let $\alpha \in \FN_M$. By definition of the field of norms \cite{Benois}, $\alpha = (\alpha_F)$ consists of elements $\alpha_F \in F$ indexed by subfields $M \supseteq F \supseteq K_\idp$ finite over $K_\idp$ that are compatible under norms. 
For each $n$, recall the decomposition $D_{M, n} = \coprod_l Gal(M_{n-l}/K_\idp)$ (cf.\ Proposition \ref{local structure of D_M, n}); then, we shall define $\hat{\alpha}: D_{M, n} \rightarrow \bar{K}_\idp$ for each $\alpha \in \FN_M$ by setting the value of $\hat{\alpha}$ at $\sigma_l \in Gal(M_{n-l}/K_\idp) \subseteq D_{M, n}$ as:
\begin{eqnarray}
 \hat{\alpha} (\sigma_l)  &:=& \sigma_l (\alpha_{M_{n-l}}).
\end{eqnarray}
This function $\hat{\alpha}: D_{M, n} \rightarrow \bar{K}_\idp$ is clearly $\bar{K}_\idp$-valuled and $G_\idp$-equivariant. In fact, by Theorem \ref{DG_p by fields of norms} and the origin of $\hat{\alpha}$, we can actually prove that these $\hat{\alpha}$ generate all algebraic $\idp$-typical Witt vectors over $K_\idp$. In this way, we could give yet another classification of $\Lambda_{O_\idp}$-rings in $\C_\idp$, which is in a sense dual to Theorem \ref{DG_p by fields of norms}. 
\end{rem}

\section{Global quotient of $DG_K$}
\label{s5}
As proved above, the structure of $DG_K$ can be well understood at least at each local level; but due to the fact that we impose no constraint on the relationship among the local factors $Q_v$ (except for those over the same prime $\idp \in P_K$, that is, $\sigma \cdot \psi_v = \psi_{\sigma(v)} \cdot \sigma$), $DG_K$ has so many finite quotients that hardly admits a thorough description of its semigroup structure (say, classification of idempotents, $\HH$-classes) unlike the case of $DR_K$. This is a major reason why we call $DG_K$ as a \emph{free} deformation of $G_K$; and gives rise to an independent problem of imposing a ``right constraint'' among the local factors $Q_v$ such that the quotient of $DG_K$ with that constraint gets an arithmetically good property at the level of global fields too. 

To be specific, we return to our study on the finite monoids $D_{L/K}$ discussed in Remark \ref{nonabelian DR_f} and constructed in \S \ref{s3.2}. We see that the finite monoids $D_{L/K}$ constructed in the proof of Theorem \ref{DG_K deforms G_K} constitute an inverse system, whose limit $\D G_K$ can admit the complete description of idempotents and their $\HH$-classes as for $DR_K$. That is, (1) $\D G_K^\times$ is still isomorphic to $G_K$; (2) the idempotents of $\D G_K$ bijectively correspond to subsets of $P_K$; (3) the maximal Galois groups $G_{K, S}$ with restricted ramifications appear as the maximal closed subgroups at the idempotents corresponding to $S$; and (4) every maximal closed subgroups at idempotents of $\D G_K$ is of this form. By the fact that $\D G_K$ is a quotient of $DG_K$, the finite $\D G_K$-sets all belong in the category $\cl DG_K$ of finite $DG_K$-sets, whose structures are well-understood as seen in \S \ref{s4}. We shall not characterize which finite $DG_K$-sets are $\D G_K$-sets; our goal in this section is to exemplify the fact that such a quotient $\D G_K$ of $DG_K$ can exist\footnote{Another possible approach to gain a good quotient of $DG_K$ is to construct a family of locally constant $\bar{K}$-valued $G_K$-equivariant functions on $DG_K$ of geometric origin such that their \emph{common domain $D$ of quotient} $DG_K \twoheadrightarrow D$ (i.e.\ all such functions factor through this quotient) admits a good (geometric) description, yet we still have $D^\times = G_K$; see our previous work \cite{Uramoto21, Uramoto23} for some concrete case study in this direction. Indeed, $DG_K$ is topologically and \emph{canonically} generated by the Frobenius elements $\psi_v$; and in our previous case \cite{Uramoto21, Uramoto23}, the Frobenius element $\psi_\idp \in DR_K$ admits a geometric interpretation in that $\psi_\idp$ acts on (the torsion points of) a complex torus $\comp^g/\Phi(\ida)$ as the isogenous quotient $\comp^g/\Phi(\ida) \twoheadrightarrow \comp^g/\Phi(\eta(\idp)^{-1} \ida)$ (cf.\ \S 3 \cite{Uramoto23}); this geometric operation itself makes sense independently of Galois groups; but via mod-$\idp$ reductions, it could be related to the Frobenius $\sigma_\idp \in G_K^{ab}$. In general, if some $DG_K \twoheadrightarrow D$ could admit a good geometric description, so would its unit group $D^\times = G_K$, whence the $\HH$-class structure of $D$ would represent the ramification-theoretic structure of $G_K$ too, which is surely the very subject of non-abelian class field theory. We did not discuss this geometric direction in this paper; but focused on developing more semigroup-theoretic ingredient of our subject.}.

To this end, let us start with recalling the definition of $D_{L/K}$ (cf.\ Theorem \ref{DG_K deforms G_K}):

\begin{defn}[finite monoids $D_{L/K}$]
For a finite Galois extension $L/K$, define $\idf:= \idf_{L/K} \in I_K$ as follows: for each $\idp \in P_K$, 
\begin{eqnarray}
 v_\idp(\idf) &:=& \min \{e \in \nat \mid Gal^e(L_\idP/K_\idp) = 1\};
\end{eqnarray}
where $\idP \in P_L$ is a prime of $L$ over $\idp$; and $Gal^e(L_\idP/K_\idp)$ is the ramification subgroup of $Gal(L_\idP/K_\idp)$ in the upper numbering; since $L/K$ is Galois, the value $v_\idp(\idf)$ does not depend on the choice of $\idP\mid \idp$. For each $\idd \mid \idf$, let us denote by $L^{(\idf/\idd)}$ the subfield of $L/K$ fixed by the composition of the ramification subgroups $Gal^{v_\idp(\idf/\idd)}(L_\idP/K_\idp)$, where $\idP$ ranges over all $\idP \mid \idp$; and $\idp$ ranges over $P_K$; then, $L^{(\idf/\idd)}$ is Galois over $K$, so that $Gal(L^{(\idf/\idd)}/K)$ makes sense. 

Then, the underlying set of our target monoid $D_{L/K}$ is defined as follows:
\begin{eqnarray}
 D_{L/K} &:=& \coprod_{\idd \mid \idf} \{\idd\} \times Gal(L^{(\idf/\idd)}/K); 
\end{eqnarray}
which is equipped with a monoid structure given by the following multiplication:
\begin{eqnarray}
 (\idd, \sigma) \cdot (\idd', \sigma') &:=& (\idd'', \sigma \cdot \sigma');
\end{eqnarray}
where $\idd'' := (\idf, \idd \cdot \idd')$ and the product of $\sigma \in Gal(L^{(\idf/\idd)}/K)$ and $\sigma' \in Gal(L^{(\idf/\idd')}/K)$ is taken in the smaller group $Gal(L^{(\idf/\idd'')}/K)$, hence well-defined. 
\end{defn}

\begin{lem}[inverse system of $D_{L/K}$]
For any finite Galois extensions $L' \supseteq L \supseteq K$, there exists a canonical surjective homomorphism $r_L^{L'}: D_{L'/K} \twoheadrightarrow D_{L/K}$, with which the finite monoids $\{ D_{L/K} \}_L$ constitute an inverse system of finite monoids. 
\end{lem}
\begin{proof}
Let $\idf, \idf' \in I_K$ be the ideals constructed above for $L, L'/K$ respectively; then by the transitivity of ramification subgroups (in upper numbering) under quotient, we deduce that $\idf \mid \idf'$. To construct a surjective homomorphism $r^{L'}_L: D_{L'/K} \twoheadrightarrow D_{L/K}$, let us write as $(\idd, \sigma)_L \in D_{L/K}$ with the subscript $L$ that indicates the domain of the element. Then:
\begin{eqnarray}
 D_{L'/K} &\rightarrow& D_{L/K} \\
 (\idd', \sigma')_{L'} &\longmapsto& ((\idd', \idf), \sigma')_L;
\end{eqnarray}
where $(\idd', \idf)$ denotes the greatest common divisor of $\idd' \mid \idf'$ and $\idf$. Note that, for each $\idd' \mid \idf'$, we have the inclusion $L^{(\idf/ (\idd', \idf))} \subseteq L^{(\idf'/\idd')}$ by the transitivity of ramification subgroups and by $\idf \mid \idf'$; hence, the map $r^{L'}_L$ is well-defined, which clearly gives a surjective monoid homomorphism. It is also easy to see that $\{D_{L/K}\}_L$ forms an inverse system with respect to the extensions $L' \supseteq L \supseteq K$, hence the claim.
\end{proof}

\begin{defn}[$\D G_K$] 
We define the profinite monoid $\D G_K$ as the following inverse limit:
\begin{eqnarray}
 \D G_K &:=& \lim_{\leftarrow L/K} D_{L/K}. 
\end{eqnarray}
\end{defn}

\begin{prop}[$\D G_K^\times = G_K$]
The unit group of $\D G_K$ is isomorphic to $G_K$.
\end{prop}
\begin{proof}
This follows from the fact that $D_{L/K}^\times = Gal(L^{(\idf)}/K) = Gal(L/K)$. 
\end{proof}

\begin{lem}[$DG_K \twoheadrightarrow \D G_K$]
There exists a surjective homomorphism $DG_K \twoheadrightarrow \D G_K$.
\end{lem}
\begin{proof}
It suffices to see that the finite $DG_K$-sets $D_{L/K}$ are galois objects of $\cl DG_K$. It is easy to see that $D_{L/K}$ are rooted objects with $(1, 1) \in D_{L/K}$ its root (cf.\ \S 3 \cite{Uramoto17}); thus, we need to show that the associated map $\End(D_{L/K}) \ni f \mapsto f^* (1, 1) \in D_{L/K}$ is surjective. But this can be easily seen as follows: For each $(\idd, \sigma) \in D_{L/K}$, consider its \emph{left} action on $D_{L/K}$, which gives a $DG_K$-equivariant endomorphism of $D_{L/K}$. Put it as $f \in \End(D_{L/K})$; we have $f^* (1, 1) = (\idd, \sigma)$, hence the claim. 
\end{proof}

\begin{prop}[$(\D G_K)^{ab} = DR_K$]
The maximal abelian quotient of $\D G_K$ is isomorphic to $DR_K$. 
\end{prop}
\begin{proof}
Notice that, for a finite abelian extension $L/K$, we have an isomorphism $D_{L/K} \simeq DR_{L/K}$ as proved in Remark \ref{nonabelian DR_f}. With this in mind, the claim follows from Theorem \ref{abelian quotient of DG_K} and the above lemma that $\D G_K$ is a quotient of $DG_K$. 
\end{proof}

\begin{thm}[classification of idempotents of $\D G_K$]
\label{classification of idempotents of DDG_K}
We have a bijective correspondence between the idempotents of $\D G_K$ and the subsets of $P_K$. 
\end{thm}
\begin{proof}
To see this, we first describe the idempotents of the finite monoid $D_{L/K}$. Let $\idf = \idf_{L/K} \in I_K$ as above for $L/K$. By definition, note that $(\idd, \sigma) \in D_{L/K}$ is an idempotent if and only if $(\idd^2, \idf) = \idd$ and $\sigma = 1 \in Gal(L^{(\idf/\idd)}/K)$; in particular, the former is to say that $v_\idp(\idd) = 0$ or $v_\idp (\idf)$ for all $\idp \in P_K$. This can be rephrased as the fact that the idempotents of $D_{L/K}$ are precisely of the form $(\ide_S, 1)$, where $S \subseteq P_K$ ranges over subsets of the prime factors of $\idf$ (= primes ramified in $L/K$); and $\ide_S \mid \idf$ denotes the divisor of $\idf$ determined by $v_\idp(\ide_S) = 0$ for $\idp \in S$, and $v_\idp(\ide_S) = v_\idp(\idf)$ otherwise. In the following, we denote by $P_{L/K}$ the (finite) set of the primes ramified in $L/K$; with this notation, our discussion above amounts to say that the idempotents of $D_{L/K}$ bijectively correspond to subsets of $P_{L/K}$, i.e.\ by $P_{L/K} \supseteq S \mapsto e_{L, S} := (\ide_S, 1) \in D_{L/K}$. 

To proceed further, let $L' \supseteq L \supseteq K$ be finite Galois extensions. By construction, we easily see that, for any $S' \subseteq P_{L'/K}$, we have $r^{L'}_L (e_{L', S'}) = e_{L, S' \cap P_{L/K}}$.  Thus, if $e \in \D G_K$ is an idempotent, its projections $e_L \in D_{L/K}$ for finite Galois extensions $L/K$ have some (unique) $S_{e, L} \subseteq P_{L/K}$ such that $S_{e, L} \subseteq S_{e, L'}$ for all $L \subseteq L'$. With this in mind, we define $S_e \subseteq P_K$ by their limit:
\begin{eqnarray}
 S_e &:=& \bigcup_{L/K} S_{e, L}. 
\end{eqnarray}
Conversely, for any $S \subseteq P_K$, we define an idempotent $e_S \in \D G_K$ by:
\begin{eqnarray}
 e_S &:=& (e_{L, S \cap P_{L/K}})_L \in \D G_K \subseteq \prod_L D_{L/K}. 
\end{eqnarray}
To see the bijective correspondence, we prove that (1) $e_{S_e} = e$ for any idempotent $e \in \D G_K$; and (2) $S_{e_S} = S$ for any $S \subseteq P_K$. For the first claim, we need to see $(\bigcup_{L'} S_{e, L'}) \cap P_{L/K} = S_{e, L}$, which follows from $S_{e, L'} \cap P_{L/K} = S_{e, L}$.    The second claim follows easily from $\bigcup_{L/K} (S \cap P_{L/K}) = S$. This completes the proof. 
\end{proof}

\begin{lem}[$\HH$-classes are $G_K$-orbits]
The $\HH$-classes of $\D G_K$ are precisely the $G_K$-orbits; also the $\JJ$-classes coincide with the $\HH$-classes in $\D G_K$. 
\end{lem}
\begin{proof}
We first see that the $\HH$-classes of $D_{L/K}$ are the $G_K$-orbits for finite Galois extensions $L/K$. In fact, if two elements $x, y \in D_{L/K}$ are $\HH$-equivalent, they are $\JJ$-equivalent in particular. Therefore, we have $a, b, c, d \in D_{L/K}$ such that $a x b = y$ and $c y d = x$; specifically, if $x = (\idd, \sigma)$ and $y = (\idd', \sigma')$, we need to have $\idd = \idd'$ in particular by definition of the monoid structure of $D_{L/K}$, and hence $x, y$ belong to the same $G_K$-orbit. The converse is easier. Notice that this also proves that the $\JJ$-classes of $D_{L/K}$ coincide with the $\HH$-classes.

To prove the theorem, let $x, y \in \D G_K$ be $\JJ$-equivalent, whence their projections $x_L, y_L \in D_{L/K}$ under $\D G_K \twoheadrightarrow D_{L/K}$ are also $\JJ$-equivalent. This means that there exists $\sigma_L \in Gal(L/K)$ for each $L/K$ such that $x_L \cdot \sigma_L = y_L$ by the above argument. Let us denote by $[x, y]_L \subseteq G_K$ the closed set consisting of $\sigma \in G_K$ such that $x_L \cdot \sigma_L = y_L$, where $\sigma_L \in Gal(L/K)$ denotes the projection of $\sigma$. We know that $[x, y]_L$ is non-empty for all $L/K$; and clearly, $[x, y]_{L'} \subseteq [x, y]_L$ for any $L \subseteq L'$. Since $G_K$ is compact, we deduce that $\bigcap_L [x, y]_L$ is non-empty. Taking $\sigma \in \bigcap_L [x, y]_L$, we have $x \cdot \sigma = y$, thus $x, y$ have the same $G_K$-orbit. The converse is easy. Again, this actually proves that the $\JJ$-classes are the $\HH$-classes. This completes the proof. 
\end{proof}

\begin{thm}[description of regular $\HH$-classes]
Let $e = e_S \in \D G_K$ be the idempotent for $S \subseteq P_K$; then the regular $\HH$-class of $e$ is isomorphic (as a profinite group) to the maximal Galois group $G_{K, S}$ with the ramifications restricted on $S$. 
\end{thm}
\begin{proof}
By the above lemma, the $\HH$-class $H_e$ is the $G_K$-orbit of $e$, that is, $H_e = e \cdot G_K$. We first see that the isotropy group of $e$ is the kernel of $G_K \twoheadrightarrow G_{K, S}$. By definition, we have $e = (e_{L, S \cap P_{L/K}})$; denoting $S_L := S \cap P_{L/K}$ for short, $e_{L, S_L} \in D_{L/K}$ belongs to the component $Gal(L^{(\idf/\ide_{S_L})}/K)$, where $\ide_{S_L} \mid \idf = \idf_{L/K}$ is characterized by $v_\idp (\ide_{S_L}) = 0$ if $\idp \in S_L$, and $v_\idp (\ide_{S_L}) = v_\idp (\idf)$ for $\idp \not \in S$. Therefore, it follows in particular that $\idp \not \in S$ is unramified in $L^{(\idf/\ide_{S_L})}/K$; and that $\sigma \in \ker (G_K \twoheadrightarrow G_{K, S})$ fixes $e$ in that $e \cdot \sigma = e$. Conversely, suppose that $\sigma \in G_K$ fixes $e$; and consider $L/K$ where those $\idp \not \in S$ are unramified. Then, applying the above fact $e_{L, S_L} \in Gal(L^{(\idf/\ide_{S_L})}/K)$ to such an $L/K$ in particular, we can deduce $\sigma = 1$ in $Gal(L/K)$ because we now have $\ide_{S_L} = 1$. Taking the limit of $L/K$ where those $\idp \not\in S$ are unramified, it follows that $\sigma$ is in the kernel of $G_K \twoheadrightarrow G_{K, S}$. 
Therefore, we can define the map $e \cdot G_K \ni e \sigma \mapsto \sigma \in G_{K, S}$, which is indeed bijective, preserves the identity elements (that is, $e \mapsto 1$), and multiplicative (because $e$ is an idempotent and commutes with $G_K$). This completes the proof. 
\end{proof}

\appendix
\section{Generalities on semigroup theory}
\label{a1}
As we introduced deformations of profinite groups (\S \ref{s1}) so that it makes sense for arbitrary profinite groups and profinite monoids, it would be helpful to review here some generalities on the theory of (profinite) semigroups, including the concepts of \emph{Green equivalence relations} and their equivalence classes in particular. Indeed, these concepts give us a basic methodology to decompose a semigroup into smaller parts, which have an arithmetic interpretation in the case of $DR_K$ (or $DR_\idf$) in particular (\S \ref{s2.2}). For our review of $DR_K$ in \S \ref{s2.2}, we actually need only a little part of the full general results below; in particular, our major concern here is to see how a monoid can contain several subgroups around its idempotents. That said, the general theory below will help those reader unfamiliar with semigroup theory to gain some intuition about what semigroups look like in general. 

In general, a monoid $D$ contains the unit group $D^\times$ as its \emph{submonoid} in that the identity element of $D^\times$ is equal to that of $D$; but $D$ may also contain more subgroups as its \emph{subsemigroups}, i.e.\ those which are closed under the multiplication of $D$, have identity elements, and form groups, but their identity elements may be \emph{idempotents} of $D$, not equal to the identity element of $D$ in general. For short, we shall mean by a \emph{subgroup at an idempotent $e \in D$} a subsemigroup of $D$ which contains $e$ and forms a group in itself with $e$ the identity element.

To see how these subgroups appear in monoids, it is helpful to decompose monoids by Green's equivalence relations. Let $D$ be a monoid (by which we mean a finite/profinite monoid throughout below). For each element $x \in D$, we denote by $DxD$ (resp.\ $Dx/xD$) the two-sided (resp.\ left/right) ideal of $D$ generated by $x$. Then we mean here by \emph{Green's quasi-ordering} and \emph{Green's equivalence relations on $D$} the following quasi-orderings and equivalence relations on the elements of $D$ defined in terms of the inclusion and equality among the (principal) ideals they generate: 

\begin{defn}[Green's quasi-ordering relations]
We define the following quasi-orderings $\leq_\JJ, \leq_\RR$ and $\leq_\LL$ on $D$: for any $x, y \in D$, 
\begin{eqnarray}
x \leq_\JJ y &\Leftrightarrow&  DxD \subseteq DyD; \\
x \leq_\RR y &\Leftrightarrow&  xD \subseteq yD; \\
x \leq_\LL y &\Leftrightarrow&  Dx \subseteq Dy.
\end{eqnarray}
\end{defn}

\begin{defn}[Green's equivalence relations]
Furthermore, we define the following equivalence relations $\JJ, \RR, \LL$ and $\HH$ on $D$: for any $x, y \in D$, 
\begin{eqnarray}
x \JJ y &\Leftrightarrow&  DxD = DyD; \\
x \RR y &\Leftrightarrow&  xD = yD; \\
x \LL y &\Leftrightarrow&  Dx = Dy;\\
x \HH y &\Leftrightarrow& x \LL y \wedge x \RR y. 
\end{eqnarray}
The $\JJ$-, $\RR$-, $\LL$-, and $\HH$-equivalence classes containing $x \in D$ are denoted by $J_x, R_x, L_x, H_x \subseteq D$.
\end{defn}

\begin{prop}[$\JJ = \LL \circ \RR = \RR \circ \LL$]
We have the equality $\RR \circ \LL = \LL \circ \RR$ of relations; and when $D$ is a profinite monoid, we further have $\JJ = \LL \circ \RR = \RR \circ \LL$. 
\end{prop}
\begin{proof}
The composition relation $\RR \circ \LL$ is defined by $x \RR \circ \LL z$ if and only if we have $x \RR y$ and $y \LL z$ for some $y \in D$; and similarly $\LL \circ \RR$. The first claim implies in particular that these (equal) composition $\LL \circ \RR = \RR \circ \LL$ actually defines an \emph{equivalence} relation of $D$; and when $D$ is profinite, the second claim says that these composition coincides with $\JJ$. See e.g.\ Propositions 2.1.3, 2.1.5 \cite{Howie}. 
\end{proof}

\begin{rem}[egg-box picture]
This proposition provides us a useful picture to describe $\JJ$-classes, known as their \emph{egg-box pictures}. Let $J \subseteq D$ be a $\JJ$-class of a monoid $D$. Then, $J$ can be decomposed as the disjoint unions of $\RR$-classes and $\LL$-classes because $\RR$ and $\LL$ are finer than $\JJ$. Thus, for some $\RR$-classes $R_i$ and $\LL$-classes $L_j$, we have:
\begin{equation}
 J = \coprod_i R_i = \coprod_j L_j
\end{equation}
The above proposition implies in particular that $R_i$ and $L_j$ have a \emph{non-empty} intersection for \emph{any} $i, j$, which is an $\HH$-class and we denote by $H_{ij} = R_i \cap L_j \neq \emptyset$. With this in mind, it is useful to describe $J$ as the following matrix-like picture decomposing $J$ as $J = \coprod_{i,j} H_{ij}$, which is called the \emph{egg-box picture of $J$}:
\begin{equation}
\begin{tabular}{|c|c|c|c|}
\hline
 & & &  \\ \hline
 & & &  \\ \hline
 & & & \\  \hline
\end{tabular}
\end{equation}
Here each box represents an $\HH$-class $H_{ij}$, where its rows represent $\RR$-classes $R_i$, while its columns represent $\LL$-classes $L_j$. (See e.g.\ \cite{Pin} for more nicer pictures.)
\end{rem}

\begin{rem}[poset of $\JJ$-classes]
Whereas each $\JJ$-class can be described as its egg-box picture, the set of all $\JJ$-classes of a monoid $D$ has a natural structure of \emph{poset} with respect to the quasi-ordering $\leq_\JJ$: That is, for two $\JJ$-classes $J, J' \subseteq D$, we define $J \leq J'$ if and only if $x \leq_\JJ x'$ for some $x \in J, x' \in J'$. By definition, the relation $J \leq J'$ does not depend on the choice of $x, x'$; and $J \leq J'$ and $J' \leq J$ imply $J = J'$. In this paper, this poset structure is studied in \S \ref{s2.2} for the case of the Deligne-Ribet monoid $DR_K$ in more detail.
\end{rem}

With these concepts in mind, we return to our original concern on subgroups in a monoid $D$: As discussed above too, the monoid $D$ may contain several subgroups at idempotents; our concern here is to classify these subgroups. The following proposition gives us an answer, which claims that the \emph{maximal} subgroup exists at each idempotent $e$ and can be precisely described as its $\HH$-class:

\begin{prop}[maximal subgroups at idempotents]
For an idempotent $e \in D$, the $\HH$-class $H_e$ containing $e$ is the maximal subgroup at $e$, namely, any other subgroups at $e$ are subgroups of $H_e$. 
\end{prop}
\begin{proof}
It is not trivial even to see that the $\HH$-class $H_e$ is closed under the multiplication of $D$, let alone that $H_e$ forms a group, whereas the maximality is an easy consequence from the definition of $\HH$. See e.g.\ Theorem 2.2.5 and Corollary 2.2.6 \cite{Howie}. 
\end{proof}

In a sense, the existence of non-trivial idempotents (i.e.\ those not equal to the identity element) distinguishes monoids from groups; and thanks to this feature, monoids can have several maximal subgroups around their idempotents that are not necessarily isomorphic to each other. It is known that, however, the maximal subgroups at \emph{$\JJ$-equivalent} idempotents are isomorphic:
\begin{prop}
Let $e, e' \in D$ be idempotents belonging to the same $\JJ$-class $J$. Then the maximal subgroups $H_e, H_{e'}$ at $e, e' \in D$ are (non-canonically) isomorphic as groups. 
\end{prop}
\begin{proof}
See e.g.\ Proposition 2.3.6 \cite{Howie}. 
\end{proof}

\begin{rem}[$\JJ$-equivalent $\HH$-classes are homeomorphic]
Moreover, let $e \in D$ be an idempotent and included in a $\JJ$-class $J = J_e$; the above proposition says that $H_e \subseteq J$ is a group with $e$ the identity element. The other $\HH$-classes in $J$ may not contain idempotents and not form groups, but still isomorphic to the group $H_e$ as \emph{sets}. Indeed, this is true more generally: every $\HH$-class in a single $\JJ$-class $J$ is isomorphic to each other as sets. This isomorphism is a homeomorphism when $D$ is a profinite monoid; if $D$ is finite in particular, every $\HH$-class in a $\JJ$-class has the same number of elements. In general, an $\HH$-class is said to be \emph{regular} if it contains an idempotent; otherwise, we say that it is \emph{non-regular}.
\end{rem}

\bibliographystyle{abbrv}
\bibliography{galois_deformation_theory}
\end{document}